\documentclass[english]{amsart}
\usepackage{amsmath,setspace}
\usepackage{amssymb}
\usepackage[all]{xy}
\usepackage{xypic,amsthm,hyperref,graphicx,stmaryrd,boxedminipage,mathrsfs,manfnt,comment}
\usepackage[all]{xy}
\usepackage[mathcal]{euscript}
\usepackage{calligra}
\usepackage{sseq}

\newcommand{\sma}{\wedge}

\newcommand{\Sp}{\operatorname{Sp}}

\newcommand{\ext}{\operatorname{Ext}}

\renewcommand{\hom}{\operatorname{Hom}}

\newcommand{\id}{\operatorname{Id}}

\newcommand{\map}{\operatorname{Map}}

\newcommand{\mf}{\mathfrak}

\newcommand{\fun}{\operatorname{Fun}}
\newcommand{\mc}{\mathcal}
\newcommand{\hh}{\operatorname{HH}}
\newcommand{\mbf}{\mathbf}
\newcommand{\thh}{\operatorname{THH}}
\newcommand{\tc}{\operatorname{TC}}
\newcommand{\diag}{\operatorname{diag}}

\newcommand{\cat}{\mc{C}\text{at}}

\DeclareMathOperator*{\hocolim}{hocolim} 

\DeclareMathOperator*{\colim}{colim}
\newtheorem{thm}{Theorem}[section]
\newtheorem{lem}[thm]{Lemma}
\newtheorem{cor}[thm]{Corollary}
\newtheorem{prop}[thm]{Proposition}

\theoremstyle{definition}
\newtheorem{defn}[thm]{Definition}
\newtheorem{example}[thm]{Example}
\newtheorem{rmk}[thm]{Remark}
\newtheorem{nota}[thm]{Notation}
\newtheorem{conv}[thm]{Convention}
\newtheorem{conj}[thm]{Conjecture}
\title{Derived Koszul Duality and Topological Hochschild Homology}
\author{Jonathan A. Campbell}
\date{}
\begin{document}
\maketitle
\begin{abstract}
 Motivated by a result from string topology, we prove a duality in topological Hochschild homology ($\thh)$. The duality relates the $\thh$ of an $\mbf{E}_1$-algebra spectrum and the $\thh$ of its derived Koszul dual algebra under certain compactness conditions. The result relies on results about module categories which may be of interest on their own. Finally, we relate this result to topological field theories and outline some future work. 
\end{abstract}
\tableofcontents

\section{Introduction}

\subsection{Motivation}
In recent years, both topological field theories and algebraic $K$-theory have received much attention, and rightly so. Algebraic $K$-theory is a universal invariant of categories that contains a wealth of information. When applied to exact categories of modules it contains arithmetic information \cite{quillen, mitchell}. When applied to categories of retractive spaces \cite{waldhausenI} (this is sometimes called $A(X)$ for a space $X$) it contains  deep geometric information, for example about automorphisms of manifolds \cite{waldhausenI,waldhausenII,waldhausenIII}. Topological field theories are similarly deep invariants of manifolds that capture not only homotopical information, but geometric structure as well. 

However, these invariants can be very difficult to compute, and only the simplest cases are directly accessible. In what is almost certainly not a coincidence, a common invariant appears in both $K$-theory computations and low dimensional field theories: topological Hochschild homology. As its name suggests, topological Hochschild homology (THH) is a generalization of the Hochschild homology of associative algebras to topological rings (aka ring spectra). In fact, some of the impetus for the development of categories of ring spectra come from the necessity of such a generalization. 

Let us begin to relate $\thh$ to both $K$-theory and topological field theories. 

It has been known for some time that the $K$-theory of a ring admits a map to the Hochschild homology of the same ring, called the Dennis trace. Work of Goodwillie \cite{goodwillie} shows a close relationship between relative $K$-theory and Connes' cyclic homology \cite{goodwillie, loday}, which is a refinement of Hochschild homology. For a ring $A$, the map $K_\ast (A) \to \hh_\ast(A)$ factors through cyclic homology and the resulting map to cyclic homology is close to an equivalence.  In particular, for simplicial rings, relative $K$-theory and relative cyclic homology are rationally equivalent. Thus, Hochschild homology calculations can be exploited to gain computational information about the algebraic $K$-theory of rings. 

There is a generalization of algebraic $K$-theory to ring spectra, not just rings, giving arithmetic invariants of ring spectra. This is not just frivolous generalization --- Waldhausen's $A$-theory fits into this framework via the identity $A(X) \simeq K(\Sigma^\infty_+ \Omega X)$ \cite{EKMM} (here $\Omega X$ denotes the based loop space). Thus, computations of $K$-theory for ring spectra are crucial to understanding diffeomorphsims of manifolds. It could thus be hoped that for ring spectra there are suitable generalizations of the above results, whereby $K$-theory admits a trace to a ``topological'' Hochschild homology which can be refined to a map to a ``topological'' cyclic homology, and where the latter map is very nearly an equivalence. Goodwillie conjectured the existence of such a generalization and B\"{o}kstedt provided the necessary coherence machinery and definition \cite{bokstedt}. Shortly thereafter, B\"{o}kstedt, Madsen and Hsiang \cite{BHM} defined topological cyclic homology ($\tc$) and showed that K-theory admits a trace to $\thh$ which factors through $\tc$. Further work by Dundas and McCarthy show that relative $K$-theory and relative topological cyclic homology agree $p$-adically  \cite{ dundas, mccarthy}. This method of computing $K$-theory has been enormously successful, and at the moment is essentially the only way to compute $K$-theory. 

Topological Hochschild homology is also closely related to field theories. One way to define a field theory is contained in Atiyah's seminal paper \cite{atiyah}. A field theory is roughly a functor from a cobordism category (the objects are manifolds, and the morphisms cobordisms) to some target category where invariants live. Baez and Dolan \cite{baez_dolan} realized that this could be extended to ``higher categories'', that is, instead of just having objects and morphisms, we have morphisms between morphisms, and morphisms between morphisms between morphisms, and so on.  Motivated by this, Hopkins and Lurie \cite{LurieFT} provide a sketch of a classification of field theories for $(\infty,n)$-categories. We will not go into detail here, but given a choice of certain target category and given a field theory $F$, if the functor $F$ assigns a point to a ring $A$, then $F(S^1)$ will be $\thh(A)$ \cite{LurieFT}. So, $\thh$ is in some sense the simplest manifold invariant obtained from field theories. 

Though the invariants obtained by many topological field theories are difficult to access, there are a few long-studied low-dimensional cases, in particular Chas and Sullivan's \cite{chas_sullivan} string topology. This is a particularly illuminating theory from the perspective above, for the following reasons. Conjecturally, string topology is the field theory classified at a point by $\Sigma^\infty_+ \Omega X$, where $X$ is some simply connected space, and $\Omega$ denotes the based loop space. This makes contact with the algebraic $K$-theory of spaces, since Waldhausen's $A(X)$ is in fact $K(\Sigma^\infty_+ \Omega X)$. As mentioned above, there is a trace $A(X) = K(\Sigma^\infty_+ \Omega X) \to \thh(\Sigma^\infty_+ \Omega X)$. The latter is equivalent, by a computation of B\"{o}kstedt and Waldhausen, to $\Sigma^\infty_+ \mc{L} X$ where $\mc{L}X$ is the free loop space, $\mc{L}X = \map(S^1, X)$. This is another hint at a relationship to string topology. 

We have one more object to introduce, and that is Koszul duality. Let $A$ be an augmented algebra and $k$ a ground field. We may form the derived tensor product $k \otimes^{\mbf{L}}_A k$ which is a coalgebra with coproduct given by the augmentation. Taking the linear dual of $k \otimes^{\mbf{L}}_A k$ we obtain the Koszul dual.  In other incarnations, the bar construction is used in place of derived tensor product, but they amount to the same thing. For this reason derived Koszul duality is often called bar-cobar duality. Roughly, we'll say that augmented algebra $A, B$ are derived Koszul dual if $\operatorname{End}_A (k, k) \simeq B$ and $\operatorname{End}_B (k, k) \simeq A$. In what follows we will often used ``Koszul duality'' in place of the more cumbersome ``derived Koszul duality.'' The idea of Koszul duality first appeared in \cite{priddy} and has occurred in many, many incarnations since then. We will be using a version of Koszul duality which, to the author's knowledge, first appear in \cite{DGI}. We note that in the explanation above, we were somewhat agnostic about the category --- it could have been any symmetric monoidal category. However, in the sequel, we will work primarly in the category of spectra. In this instance, a good example of Koszul dual algebras are $\Sigma^\infty_+ \Omega X$ and $DX$, the Spanier-Whitehead dual of $X$, where $X$ is a compact, simply-connected topological space. 

Having introduced the major players, $\thh$, field theories and Koszul duality, we can proceed to tie them together: appropriately enough, we will do this with string topology. String topology, the based loop space $\Omega X$ and Hochschild homology have long been known to have a close relationship. Even before the discovery of string topology, chain-level computational results existed about $\hh_\ast (C_\ast (\Omega X))$ \cite{goodwillie}.  One of these results \cite{jones_mccleary} relates $\hh_\ast (C_\ast (\Omega X))$ and Hochschild homology of the derived Koszul dual of $C_\ast (\Omega X)$, which is the cochains $C^\ast (X)$. This is where derived Koszul duality enters the picture.

Motivated by the above considerations in string topology, Koszul duality and the algebraic $K$-theory of spaces, Ralph Cohen has asked some questions about the relationship between $\Sigma^\infty_+ \Omega X$ and $DX$,  when viewed through the lens of $K$-theory and $\thh$. The first is: What is the relationship between $A(X)$ and $K(DX)$? This question was beautifully answered in \cite{BlumbergMandellKoszul}. The second is as follows:  given the observation (noted above) that $\thh(\Sigma^\infty_+ \Omega X) \simeq \Sigma^\infty_+ \mc{L} X$ and $D(\thh(DX)) \simeq \Sigma^\infty_+ \mc{L} X$ (observed by Cohen and Jones), can one explain the equivalence of these two as an actual (rather than abstract) equivalence, and does this same relation hold for other Koszul dual algebras? This paper is an affirmative answer to the second question:  We provide an equivalence and show that a similar relationship holds for the $\thh$ of any Koszul dual $\mbf{E}_1$-algebras. Our hope is that this will provide some computational traction, as well as hint at deeper phenomona in field theories. 

For the proof, we cannot approach the problem in the same way that \cite{jones_mccleary} does in the case of chains, for we run into severe technical issues. The most vexing is that we need both bar and cobar constructions to properly deal with (derived) Koszul duality, but in the category of spectra, there appears not to be a way to make both multiplication and comultiplication into associative, resp. coassociative, operations and retain good homotopy theoretic properties. The introduction of necessary $A_\infty$-bar constructions then renders the computations much more difficult. In short, we run into coherence issues. 

Instead we take a different, more conceptual, approach. Topological Hochschild homology is Morita invariant, that is, we can define $\thh$ on the module category of an algebra $A$, and this coincides with $\thh$ of $A$. Thus, to unearth relationships about the $\thh$ of two algebras, it helps to understand the relationship between their module categories. But what category do these module categories live in? It turns out that the most convenient category to use is the category of \textit{small, stable, idempotent complete} quasicategories, denoted $\mc{C}\text{at}^{\text{perf}}_\infty$, introduced in \cite{BGT}. Very roughly, this should be thought of as some elaboration of spectral categories --- more details will be forthcoming later. The salient property of this category is that it is a natural source for both $K$-theory and $\thh$ and it is symmetric monoidal. 

With these concepts in place, we prove the following for Koszul dual $\mbf{E}_1$-algebras:

\begin{thm}
Let $A$ and $B$ be Koszul dual $\mbf{E}_1$-algebras, with $A$ perfect (i.e. in some sense finitely generated) as an $S$-module, then 
\[
D(\operatorname{Mod}^{\operatorname{perf}}_A) \simeq \operatorname{Mod}^{\operatorname{perf}}_{B^{\text{op}}}
\]
where $D$ denotes the dual in the category of (small, stable, idempotent-complete) $\infty$-categories. That is, the $(\infty,1)$-category of perfect $A$-modules is dual to the  $(\infty,1)$-category of perfect $B^{\text{op}}$-modules. 
\end{thm}

\begin{rmk}
Below, we will be more precise about what symmetric monoidal structure we are using for duality, and what duality means. Roughly, there is a monoidal structure on the $\infty$-category of $\infty$-categories, the the monoidal structure we are using is a special case of that. We will also be more precise about the compactness condition, but it approximately says that $A$ can be built out of $S$ in a finite way. 
\end{rmk}

We now have a duality between module categories, but we need to know how $\thh$ behaves with respect to such dualities. That is, in order to apply $\thh$ to our current problem, we need to know that as a functor from categories to spectra it is symmetric monoidal. This is in fact taken care of in \cite{BGT_TC}. 

\begin{prop}[\cite{BGT_TC}]
$\thh: \mc{C}\text{at}^{\operatorname{perf}}_\infty \to \Sp$ is symmetric monoidal. 
\end{prop}

Finally, combining this results we get the precise statement of the theorem. 

\begin{thm}
Suppose $A$ and $B$ are Koszul dual and that $A$ is small as an $S$-module. Then
\[
D(\thh(A)) \simeq \thh(B^{\text{op}})
\]
where $D$ denotes Spanier-Whitehead dual in the category of spectra. 
\end{thm}

The theorem and the technique of proof immediately suggest further questions. One of the most tantalizing conjectures is that there is a deeper relationship between Koszul duality and field theories (see 3 below and section 5). 

\begin{enumerate}
\item Does a similar statement hold for $\thh_R$?  Does the same relationship on module categories hold? The answer is ``yes'' and will appear in forthcoming work. 

\item What is the cyclic structure? Does this duality hold for $\tc$ as well? More specifically, given a spectrum with $S^1$ structure, the dual also has an $S^1$-structure. Is it the case that the dual of $\tc(A)$ and $\tc(B)$ have actions that agree in some way? If not, how do they fail to be equivalent? 

\item This result suggests that Koszul duality provides duality for field theories. In fact, this result is meant as a test case for a much larger conjecture. This conjecture (which Ralph Cohen introduced us to) would be that the topological chiral homology of an $\mbf{E}_n$-algebra $A$ is dual to the topological chiral homology of the Koszul dual to $A$. Roughly, given an $\mbf{E}_n$-ring spectrum, it has a Koszul dual $\mbf{E}_n$-spectrum $B$. Suppose we are also given an oriented $n$-manifold $M$.  Is it the case that $\int_M A \simeq D\left( \int_M B\right)$? See below for definitions and a more precise conjecture. 
\end{enumerate}

\subsection{Outline}

In section 2 we give the basic definitions and recall some important definitions about $\infty$-categories. 

In section 3 we discuss certain rigid models for $\infty$-categories that will be extremeley useful for us. We will also discuss some compactness conditions.

In section 4 we tie all of the above together and prove the main result. 

In section 5 we present some (almost certainly true) conjectures, and then mention the relationship to topological field theories (in the sense of Hopkins-Lurie, i.e. $(\infty,n)$-functors). A consequence of the work above is that Koszul duality is in some sense the appropriate notion of duality for field theories. 

\subsection{A Note on the Use of $\infty$-categories}

The choice to work with $\infty$-categories was inspired by the pleasing form the result takes in this context. Also, we work heavily with module categories and $\infty$-categories provide the correct language to encode properties of module categories. After all, almost all work with module categories typically involves passing to some ``derived'' setting. 

The choice to work primarily with quasicategories was a pragmatic one --- the work could have been completed in any developed model of $(\infty,1)$-categories. However, Joyal and Lurie have thoroughly developed quasicategories as a model for $(\infty,1)$-categories, and we have all the necessary tools of category theory (limits, colimits, Grothendieck constructions, compactness, etc) at our disposal thanks to them.

\subsection{Notation and Conventions}

\begin{conv}\label{coh_nerve}
We note that ``$\infty$-category'' is a shortening of ``$(\infty,1)$-categories'', which is in turn a catch-all term for any one of various model categories for the homotopy theory of homotopy theories (\cite{Bergner},\cite{Rezk},\cite{LurieHTT},\cite{BarwickKan}, \cite{Joyal}).

We will try to carefully distinguish between various models of $(\infty,1)$-categories, however, the default model will be quasicategories. We will say \textit{quasicategory} when we mean a simplicial set with inner horn fillers. We will say \textit{relative category} when we mean a category equipped with weak equivalences. We will say \textit{simplicial category} when we mean a category enriched in simplicial sets (which is \textit{not} the same as a simplicial object in categories). Finally, \textit{spectral categories}, which are model for stable $(\infty,1)$-categories will be called such. For ease of reference, we'll list the various categories of categories:
\begin{itemize}
\item[--]$\cat_{\infty}$ - quasicategories 
\item[--]$\cat_{\text{rel}}$ - the category of small relative categories
\item[--] $\cat_{\Delta}$ - the category of small simplicial categories
\item[--] $\cat_{\Sp}$ - the category of spectral categories
\item[--] $\cat^{\operatorname{Ex}}_\infty$ - the quasicategory of stable $\infty$-categories
\item[--] $\cat^{\operatorname{perf}}_\infty$ - the quasicategory of small, stable, idempotent complete $\infty$-categories. 
\end{itemize}

These can all be arranged in the diagram below, which deserves some explanation. The downward verticle map $L^H$ is the hammock localization of Dwyer and Kan \cite{DK1}. The symbol ``$\simeq$'' indicates an equivalence of underlying homotopy theories. The left horizontal map $\mbf{N}^{\text{coh}}$ is the homotopy coherent nerve of Cordier-Porter \cite{cordier_porter}, giving an equivalence between the homotopy theory of simplicial categories, and the homotopy theory of simplicial sets with the Joyal model structure. The downward vertical arrows indicate the equivalence of homotopy theories between spectral categories equipped with triangulated equivalences (resp. Morita equivalences) and stable $\infty$-categories (resp. small stable idempotent complete $\infty$-categories) \cite{BGT}.

\[
\xymatrix{
\cat_{\text{rel}}\ar[d]^{\simeq}_{L^H} \ar[dr] &  & \cat^{\text{triang}}_{\text{Sp}} \ar[d]^\simeq & \cat^{\text{Mor}}_{\text{Sp}}\ar[d]^\simeq \\
\cat_{\Delta} \ar[r]_{\mbf{N}^{\text{coh}}}^{\simeq} & \cat_\infty & \cat^{\text{Ex}}_\infty  \ar@{_{(}->}[l] & \cat^{\text{perf}}_\infty \ar@{_{(}->}[l]
}
\]

For later use, we include a short discussion of the passage from relative categories (or model categories) to quasicategories. Given a model category $\mbf{M}$ one can consider it as a relative category by forgetting the fibrations and cofibrations. We can then form the Dwyer-Kan localization $L^H \mbf{M}$ which is a simplicial category. Fibrantly replace this simplicial category to obtain another simplicial category $(L^H \mbf{M})'$ whose mapping spaces are Kan complexes. Then apply the coherent nerve $\mbf{N}^{\text{coh}}(L^H \mbf{M})'$ to obtain a quasicategory. 

Another procedure is to consider the subcategory $\mbf{M}^c$ of cofibrant objects of $\mbf{M}$ and weak equivalences, $W$, as a marked simplicial set $(\mbf{M}^c, W)$ (see \cite[Ch. 3]{LurieHTT}). By taking the fibrant replacement of the marked simplicial set $(\mbf{M}^c, W)$, we obtain a simplicial set $\mbf{M}^c [W^{-1}]$ which is a quasicategory and a model for the localization. 

For a model category that is already a simplicial model category, one can pass to a quasicategory by simply considering the subcategory $\mbf{M}^\circ$ of fibrant-cofibrant objects and then taking the coherent nerve $\mbf{N}^{\text{coh}}(\mbf{M}^{\circ})$. One may then ask if this agrees with the procedure in the previous paragraph. In fact, by work of Dwyer and Kan \cite{DK3} it does (this fact is proved in a slightly different context in \cite[1.3.4.20]{LurieHA}). 

All localization procedures above produce categorically equivalent quasicategories, so we are free to use any such procedure.

\end{conv}

\subsection{Acknowledgements}

The author would like to thank his advisors, Ralph Cohen and Andrew Blumberg for their patience, encouragement, technical help, inspiration, and general wisdom. He would also like to thank Clark Barwick, John Lind, Anna Marie Bohmann and Cary Malkiewich for helpful conversations. Finally, he is grateful to Matt Pancia for a very careful reading which caught many typos and expositional issues. 

\section{Koszul Duality and Compactness Conditions}

In this section we give the basic definition of Koszul duality, as well as how it relates to categories of modules. We'll demonstrate that Koszul duality of $\mbf{E}_1$-algebras is, similar to Morita equivalent algebras, something that is detected at the level of module categories. This characterization is what allows for the main result of the paper. 

Excellent treatments of Koszul duality in this context are given in \cite{DGI} and \cite{BlumbergMandellKoszul}.

\subsection{Koszul Duality: Basic Definition}

We start with a few definitions. For the moment we assume we are working in a modern category of spectra, e.g. \cite{EKMM}. We pick the model of Elmendorff-Kriz-Mandell-May for technical convenience and the fact that key previous work on derived Koszul duality \cite{BlumbergMandellKoszul} uses this category. 

Classically, Koszul duality is defined for (differential graded) associative algebras over a field. However, we want to work with spectra and so we need a sensible analogue of both associative algebra and a ground field. The former is provided by the definition and remark below. For a ground ``field'' we take the sphere spectrum, $S$. However, we note that we could have chosen \textit{any} ring spectrum and the analysis below would work.

\begin{defn}
An \textbf{$\mbf{E}_1$-ring spectrum} is a spectrum that allows an action by the $\mbf{E}_1$-operad. \cite{may_operads}
\end{defn}

\begin{rmk}
In any modern model category of spectra, $\mbf{E}_1$-spectra can be rectified to be associative on-the-nose \cite{EKMM}. 
\end{rmk}

\begin{defn}
A ring spectrum  $A$ is \textbf{augmented} if it is equipped with a ring map to the sphere spectrum $A \to S$. 
\end{defn}
\begin{rmk}
This augmentation is absolutely crucial: it gives $S$ an $A$-module structure, and we could not state our results without it. The fact that $S$ is an $A$-module will be used repeatedly below, and in fact the compactness or non-compactness of $S$ as an $A$-module will play an important role. 
\end{rmk}

Before we proceed, some notation:

\begin{nota}
For lack of better notation, Lurie's notation for the Koszul \cite{LurieDAGX} dual will be co-opted. For an augmented algebra $A$, we let $\mf{D}A$ denote the Koszul dual. Furthermore, in the sequel $\operatorname{RHom}$ will always denote derived homomorphisms. The derived homomorphisms between $X$ and $Y$ in a model category are computed by cofibrantly replacing $X$ (typically denoted $QX$) and fibrantly replacing $Y$ (typically denoted $RY$) and computing $\operatorname{Hom}(QX, RY)$. In the situation we are working in, EKMM spectra, all objects are fibrant, and so fibrant replacement is unnecessary. 
\end{nota}

Since we are trying to generalize a classical definition, we should at least examine that definition. In Koszul duality (sometimes called Bar-Cobar duality, \cite{HMS}) for algebras, suppose we have an augmented $k$-DGA $M$. We can consider $\mf{D}(M) = \operatorname{RHom}_M (k, k) \simeq \hom( B(k, M, k), k)$, which is also a $k$-algebra, so we can consider $\operatorname{RHom}_{\mf{D}M} (k, k)$. Under some hypotheses, we will have $M \simeq \operatorname{RHom}_{\mf{D}M}(k, k)$. In this case $M$ and $\mf{D}M$ are said to be Koszul dual. 

The goal is to now generalize this construction to ring spectra. It will end up taking a slightly different (though equivalent) form. 

If we are to generalize the construction above, we should have that augmented ring spectra $A$ and $B$ are Koszul dual if $B \simeq \operatorname{RHom}_A (S, S)$. However, we don't just want an abstract equivalence, we want a map. By adjointness, in order to have such a map $B \to \operatorname{RHom}_A (S, S)$, it is necessary that we have a map 
\[
A \sma B \to S
\]
that extends the augmentation of both $A$ and $B$. Another way of saying this is that we want the $A$-module and $B$-module structures on $S$ to commute. Furthermore, in order for something to truly be a dual, we would want $A \simeq \operatorname{RHom}_B (S, S)$ as well. This leads us to the following definition. 

\begin{defn}\cite{DGI}
 The ring spectrum $A$ gives a map $S \to S$ via the $A$-module structure on $S$ and this map is obviously $\operatorname{End}_A (S,S)$-equivariant. Thus, we have a natural map
\[
A \to \operatorname{REnd}_{\operatorname{REnd}_A (S,S)} (S,S)
\]
called the \textbf{double centralizer}. $A$ is \textbf{dc-complete} if that map is an equivalence. 
\end{defn}

\begin{rmk}
Here, and in other places, we use the terminology (e.g. dc-complete) of \cite{DGI}. 
\end{rmk}

We are now able to present the definition of (derived) Koszul duality that we will use:

\begin{defn}
Augmented ring spectra $A$ and $B$ are \textbf{Koszul dual} if

\begin{enumerate} 
\item  The $A$ and $B$-module structures induced by augmentations commute. More precisely we have a commutative diagram
\[
\xymatrix@R=.4cm{
A \sma B \sma S \ar[dd]_{\operatorname{tw} \sma \id} \ar[r] & A \sma S \ar[dr] & \\
& & S \\
B \sma A \sma S \ar[r] & B \sma S \ar[ur] & 
}
\]
\item The map $A \to S$ is dc-complete. 
\end{enumerate}
\end{defn}

\begin{rmk}
The commutativity of the diagram above is the exact version of the intuition that we want a map $B \to \operatorname{REnd}_A (S, S)$ --- such a map exists when the module structures commute. 
\end{rmk}

\begin{rmk}
We would like to emphasize that there are many definitions of what it means for two objects to be Koszul dual. In particular, there are definitions for operads  and also other types of algebra, e.g. $\mbf{E}_n$-algebras \cite{HMS, GinzburgKapranov, FresseEn, LurieDAGX}. They all have a flavor similar to the above, and typically involve bar-cobar constructions. 
\end{rmk}

\begin{rmk}
A reasonable, though rough,  way to think about this is that $\mbf{E}_1$-ring spectra $A$ and $B$ are Koszul dual if
\begin{align*}
\ext_A (S, S) &\simeq B\\
\ext_B (S, S) &\simeq A
\end{align*}
which is quite similar to the case of algebras. What keeps this from being an exact definition is the lack of required maps.
\end{rmk}

\begin{example}
From the introduction it should be clear that our motivating example is $A = \Sigma^\infty_+ \Omega X$ and $B = DX$ where $X$ is simply connected. Particularly nice point-set models of this case are discussed in \cite{BlumbergMandellKoszul}. We will not give complete details as to why these are Koszul dual, but it is easy to sketch. If we assume infinite suspension, $\Sigma^\infty_+$, commutes with the bar construciton, then, since the two-sided bar construction models derived smash product,
\begin{align*}
\operatorname{REnd}_{\Sigma^\infty_+ \Omega X} (S, S) &\simeq \hom(S \sma^{L}_{\Sigma^\infty_+ \Omega X} S, S) \simeq \hom(B(S, \Sigma^\infty_+\Omega X, S), S)\\
&\simeq \hom(\Sigma^\infty_+ B(\ast, X, \ast), S) \simeq \hom(X, S) := DX
\end{align*}
Similarly, if we assume the Spanier-Whitehead dual commutes with the bar construction then 
\begin{align*}
\operatorname{REnd}_{DX} (S, S) &\simeq \hom(S \sma^L_{DX} S, S) \simeq \hom(B(S, DX, S), S)\\
&\simeq  \Sigma^\infty_+ B(\ast, X, \ast) \simeq \Sigma^\infty_+ \Omega X
\end{align*}
\end{example}

\begin{example}
The map $\mbf{S}_p \to H\mbf{F}_p$ is dc-complete \cite{DGI}. As noted in \cite{DGI} this should be related to the convergence of the Adams spectral sequence, however, the author knows of no source where this is spelled out. 
\end{example}

\subsection{Compactness in triangulated, model and $\infty$-categories}

For us, the key part of Koszul duality is the relationship that arises between the category of compact modules over a ring spectrum and the category of compact modules over its Koszul dual. In order to clearly state the various relationships, we will need a discussion of some category theoretic notions.  We will discuss model and $\infty$-categorical elaborations of these notions and  relate these. This will allow us to translate between various notions of compactness --- this will be necessary to translate results proved in model categories to $\infty$-categorical results. 

We begin with the notion of compactness in ordinary categories. 

\begin{defn}\label{compact}
Let $\mc{C}$ be an (ordinary) category which admits filtered colimits and $c \in \mc{C}$ an object. Then $c$ is \textbf{compact} if morphisms out of it commute with filtered colimits. That is, if $K$ is a filtered category and $\mbf{D} : K \to \mc{C}$ is a diagram then
\[
\operatorname{Mor}(c, \colim_K \mbf{D}) \xleftarrow{\sim} \colim_K \operatorname{Mor}(c, \mbf{D})
\]
\end{defn}

The homotopy category of spectra has a great deal of structure, which will be using for definitions below. In particular, the homotopy category of spectra is triangulated.  A \textbf{triangulated} (ordinary) category is an additive category, $\mbf{C}$,  together with an  endofunctor $\Sigma: \mbf{C} \to \mbf{C}$, which is an equivalence,  and a set of \textit{distinguished triangles} $X \xrightarrow{f} Y \to Z \to \Sigma X$ satisfying certain reasonable axioms. The distinguished triangles should be thought of as fiber or cofiber sequences and the axioms reflect that intuition. For a complete definition see \cite{hovey}.

\begin{rmk}
One may think of a triangulated category as the homotopy category of a category satisfying the usual formal properties that the category of spectra or chain complexes satisfy. 
\end{rmk}

Using the triangulated structure of a category, we can define another notion of ``compactness'' which corresponds more closely to the intuition that something compact should be finitely generated. 

\begin{defn}
A \textbf{thick subcategory} of a triangulated category is a category that is closed under retracts and the formation of fibers and cofibers, in other words if it is closed under triangles. 
\end{defn}

The objects which lie in a thick subcategory, can in some sense be thought of as ``finitely generated'' as in the following example. 

\begin{example}
Let $\operatorname{Ho}(\operatorname{Sp})$ be the homotopy category of the model category of spectra (we will see below why this is triangulated). Then, we may define $\mc{T}$ to be the smallest thick subcategory containg $S$, the sphere spectrum. The category $\mc{T}$ will then be the category of finite $S$-modules, in other words, spectra built out of a finite number of cells and retracts of such. 
\end{example}

Thick subcategories are closed under certain limits and colimits, and one could imagine other categories closed under larger classes of limits and colimits. One such type of category, which will be useful in the sequel, is the following:

\begin{defn}
A \textbf{localizing subcategory} is a thick subcategory closed under arbitrary coproducts. 
\end{defn}

We now move on to defining the corresponding notions for model categories. In particular, we will be working with stable model categories, which we must define. For a much more complete discussion, see \cite{hovey}. The definition is initially due to Quillen. 

\begin{defn}
A \textbf{stable model category} $\mbf{M}$ is a pointed model category (i.e. it has an object $\ast$ which is both initial and final) such that the suspension and loop functors constructed in \cite{quillen} ($\Sigma: \operatorname{Ho} \mbf{M} \to \operatorname{Ho} \mbf{M}$ and $\Omega: \operatorname{Ho} \mbf{M} \to \operatorname{Ho} \mbf{M}$ respectively)  are inverse equivalences. 
\end{defn}

\begin{example}
The classic and motivating example is any model category of spectra. 
\end{example}

The homotopy category of any stable model category has a rich structure. Most important for us will be the following:

\begin{prop}\cite{hovey}
Let $\mbf{M}$ be a stable model category. Then $\operatorname{Ho}(\mbf{M})$ is triangulated. 
\end{prop}

The reason that this is important is that it allows us to define fibers and cofibers in the homotopy category. In general, homotopy categories need not admit such constructions. 

We define compactness in a (not necessarily stable) model category in exact analogy with compactness in an ordinary category. 

\begin{defn}
Let $\mbf{M}$ be a model category and let $K$ be a filtered category and $\mbf{D}: K \to \mbf{M}$ be a diagram in $\mbf{M}$. Then an object $m \in \mbf{M}$ is \textbf{compact} if the following map is a weak equivalence
\[
\operatorname{map}(m, \hocolim_K \mbf{D}) \xleftarrow{\sim} \hocolim_K \operatorname{map}(m, \mbf{D})
\]
Here $\operatorname{map}(-,-)$ is the DK mapping space constructed in \cite{DK3}. 
\end{defn}

One could ask how this interacts with the definition of compactness for ordinary categories. In particular, if an object $m$ is compact in $\mbf{M}$ is the image of $m$ compact in $\operatorname{Ho}(\mbf{M})$? In many cases this is true, as we will see below.  

We move on to thick and localizing subcategories of stable model categories. In the case of triangulated categories we defined some subcategory to be thick or localizing if it was closed under taking fibers and cofibers. We can certainly define this for the homotopy category of a stable model category. However, we would like a good point-set notion as well. To this end, we offer the following definition. 

\begin{defn}\label{triang_model}
Let $\mbf{M}$ be a stable model category. A subcategory $\mbf{T}$ is \textbf{thick} if it is closed under weak equivalences, homotopy cofibers, homotopy fibers and retracts. Further, $\mbf{T}$ is localizing if it is thick and closed under arbitrary homotopy coproducts. 
\end{defn}

\begin{prop}
Let $\mbf{M}$ be a stable model category and let $\mbf{T}$ be a thick subcategory of $\mbf{M}$. Then $\operatorname{Ho}(\mbf{T})$ is a thick subcategory of $\operatorname{Ho}(\mbf{M})$. Conversely, if $T$ is a thick subcategory of $\operatorname{Ho}(\mbf{M})$, then the homotopy theoretic essential image of $T$ in $\mbf{M}$ is a thick subcategory of $\mbf{M}$. 
\end{prop}

We will also be needing all of the above notions for $\infty$-categories.  Most of these definitions appear in \cite{LurieHA}, \cite{LurieHTT}, with the exception of thickness. 

\begin{defn}\cite[Def. 5.3.1.7]{LurieHTT}
An $\infty$-category $\mc{C}$ is \textbf{filtered} if for every $\omega$-small simplicial set $K$ and every functor $f: K \to \mc{C}$, $f$ can be extended to a functor on the right cone $\overline{f} : K^{\rhd} \to \mc{C}$. 
\end{defn}

\begin{defn}\cite[Def. 5.3.4.5]{LurieHTT}
Let $\mc{C}$ be an $\infty$-category and $C \in \mc{C}$ an object. Then $C$ is \textbf{compact} if the functor $j_C : \mc{C} \to \widehat{\mc{S}}$ corepresented by $C$ preserves filtered colimits.  
\end{defn}

Having collected definitions of compactness in ordinary categories, model categories and $\infty$-categories, we proceed to relate them. To do this, we need the following result, which the author first learned this in \cite{toen_vaquie}.

\begin{prop}\label{hocolim_colim}
Let $\mbf{M}$ be a compactly generated model category, $K$ a filtered category, and $\mbf{D}:K \to \mbf{M}$ a diagram. Then the image of the natural map
\[
\hocolim_{K} \mbf{D}(k) \to \colim_K \mbf{D}(k)
\]
in $\operatorname{Ho}(\mbf{M})$ is an isomorphism. 
\end{prop}

With this in place, we can compare various notions of compactness. 

\begin{thm}\label{compact_comparison}
Let $\mbf{M}$ be a compactly generated model category and let $\mc{M}$ be the corresponding $\infty$-category $\mbf{N}(\mbf{M})[W^{-1}]$ and let $\operatorname{Ho}(\mc{M})$ be the corresponding homotopy category (which is the same as $\operatorname{Ho}(\mbf{M})$). Then
\begin{enumerate}
\item An object $m \in \mbf{M}$ is compact if and only if its image in $\operatorname{Ho}(\mbf{M})$ is. 
\item An object $m \in \mbf{M}$ is compact if and only its corresponding element in $\mc{M}$ is. 
\end{enumerate}
\end{thm}
\begin{proof}
Item (1) is proved using \ref{hocolim_colim} twice, see the diagram below. The left vertical arrow is an equivalence by applying \ref{hocolim_colim} to the compactly generated model category $\mbf{M}$ and using the fact that $\map(-,-)$ preserves weak equivalences in each variable \cite{hovey}. The right vertical arrow is by applying \ref{hocolim_colim} to $\mc{S}\text{et}_{\Delta}$. 
\[
\xymatrix{
\operatorname{map}(m, \hocolim_K \mbf{D})\ar[d]_{\sim} & \hocolim_{K} \operatorname{map}(m, \mbf{D})\ar[l]_{\sim}\ar[d]_\sim\\
\operatorname{map}(m, \colim_K \mbf{D}) & \colim_K \operatorname{map}(m, \mbf{D})
}
\]
This shows $\operatorname{map}(m, \colim_K \mbf{D})$ and $\colim_K \operatorname{map}(m, \mbf{D})$ are equivalent. Of course, $\operatorname{map}(m, \colim_K \mbf{D})$ and $\colim \operatorname{map} (m, \mbf{D})$ are simplicial sets, but on $\pi_0$ it gives that $m$ is compact ($\pi_0$ commutes with colimits). 

Item (2) is proved by comparing homotopy colimits in a simplicial category (in our case the Dwyer-Kan localization $L^H \mbf{M}$) with colimits in an $\infty$-category. This is done is \cite[Th 4.2.4.1]{LurieHTT}. In order to use \cite[Th 4.2.4.1]{LurieHTT} the simplicial category must be fibrant. However, using framings \cite{hovey}, the mapping spaces can be constructed to be fibrant. 
\end{proof}

Moving on, we can only define the notion of thick and localizing categories in the presence of stability. In order to do this, we have to define stable $\infty$-categories. 

\begin{defn}\cite[Defn. 1.1.1.9]{LurieHA}
A pointed $(\infty,1)$-category is \textbf{stable} if 
\begin{enumerate}
\item Fibers and cofibers exist
\item Any cofiber sequence is a fiber sequence and \textit{vice versa}. 
\end{enumerate}
\end{defn}

In a pointed $(\infty, 1)$-category $\mc{C}$, we can define endofunctors $\Sigma: \mc{C} \to \mc{C}$ and $\Omega: \mc{C} \to \mc{C}$ \cite{LurieHA}. These allow us to define corresponding sususpension and loop endofunctors on $\operatorname{Ho}(\mc{C})$, which provide structure maps for a triangulated category:

\begin{thm}\cite[Thm. 1.1.2.14]{LurieHA}
Let $\mc{C}$ be a stable $\infty$-category, then $\operatorname{Ho}(\mc{C})$ is triangulated. 
\end{thm}

With the definition of stable $\infty$-category in place we can define thick and localizing subcategories. 

\begin{defn}
A \textbf{thick subcategory} $\mc{C}' \subset \mc{C}$ of a stable $\infty$-category $\mc{C}$ is a stable sub-category closed under retracts. A \textbf{localizing subcategory} is further closed under taking coproducts.
\end{defn}

\begin{defn}
Let $O$ be a set of objects in $\mc{C}$. The \textbf{thick subcategory generated by $O$} is the smallest thick subcategory of $\mc{C}$ containing $O$. 
\end{defn}

We would like to relate stable model categories and stable $\infty$-categories as well as thick and localizing subcategories. Regarding the first, we have the following proposition, whose proof is essentially an exercise:

\begin{prop}
Let $\mbf{M}$ be a stable model category. Then $\mbf{N}(\mbf{M})[W^{-1}]$ is a stable $\infty$-category. 
\end{prop}

\begin{rmk}
The notation $\mbf{N}(\mbf{M})[W^{-1}]$ is explained in the introduction \ref{coh_nerve}. 
\end{rmk}

We have defined compactness in sitations of increasing complexity: ordinary categories, model categories and $\infty$-categories. We have also related the definitions of compactness that occur in each of these. Now it remains to relate the definitions of thick subcategories. To avoid unnecessary generality,  we specialize to certain $\infty$-categories we will actually use. The the following will be the most important example of thick subcategories: thick subcategories of a category of modules.  We will be considering these thick subcategories in both the case of model categories and $\infty$-categories, so we will be careful to distinguish between these two cases in notation. 

\begin{conv}
In order to distinguish between model categories of modules and $\infty$-categories of modules, we will decorate the former with a superscript $M$ (for model),  $\operatorname{LMod}^M_A$ and we will omit decoration for the latter, $\operatorname{LMod}_A$. Note that $\operatorname{LMod}^M_A$ and $\operatorname{LMod}_A$ have the same homotopy theory. 
\end{conv}

\begin{example}
Let $\mc{D}_A$ be the derived category of the category of $A$-modules (again, we are working with EKMM spectra). Let $T_A$ be the thick subcategory generated by $A$ (this is the same as the compact $A$-modules). 

The thick subcategory of $\operatorname{LMod}^M_A$ generated by $A$ will be called the \textbf{category of perfect $A$-modules} and denoted $\operatorname{LMod}_A^{\circ,\operatorname{perf}}$. The symbol ``$\circ$'' means fibrant-cofibrant elements of the model category. Recall that $\operatorname{LMod}^{\circ, \operatorname{perf}}_A$ is the homtopy theoretic essential image of the inclusion  $T_A \subset \operatorname{LMod}_A$. 
\end{example}

\begin{example}
Let $\operatorname{LMod}_A$ be the $\infty$-category of $A$-modules, where $A$ is a ring spectrum, so $\operatorname{LMod}_A$ is stable.  The category $\operatorname{LMod}^{\text{perf}}_A$ is the thick subcategory generated by $A$. Note that this is the full subcategory determined by the homotopy category $\mc{T}_{\operatorname{LMod}^M_A} (A)$, i.e. the full subcategory determined by the corresponding thick subcategory in the homotopy category. 
\end{example}

\begin{example}
Let $\mc{T}^M_A (S)$ be the thick subcategory of $\operatorname{LMod}^{M,\circ}_A$ generated by $S$. Let $\mc{T}_A (S)$ be the thick subcategory of $\operatorname{LMod}_A$ generated by $S$. 
\end{example}

Having collected our main examples, we can study the relationship between them.

\begin{prop}
$\mbf{N}(\operatorname{LMod}^{M, \circ,\operatorname{perf}}_A) \simeq \operatorname{LMod}^{\operatorname{perf}}_A$, that is $\mbf{N}(\operatorname{LMod}^{M,\circ, \operatorname{perf}}_A)$ and $\operatorname{LMod}^{\operatorname{perf}}_A$ are categorically equivalent. 
\end{prop}

\begin{proof}
We know that $\mbf{N}(\operatorname{LMod}^{M,\circ}_A) \simeq \operatorname{LMod}_A$ \cite[4.3.3]{LurieHA}. That, is there is a map $\mbf{N}(\operatorname{LMod}_A^{M, \circ}) \to \operatorname{LMod}_A$ that is fully faithful and essentially surjective. By definition, the subcategory of perfect objects of $\operatorname{LMod}_A^{M,\circ}$ is the homotopy-theoretic image of the inclusion $T_A  \subset \operatorname{LMod}_A^{M, \circ}$. Thus, the full subcategory of perfect $A$-modules is identified as the pullback
\[
\xymatrix{
\operatorname{LMod}^{\text{perf}}_A \ar[d]\ar[r] & \operatorname{LMod}_A \simeq \mbf{N}(\operatorname{LMod}^{M,\circ})\ar[d]^{\text{Ho}}\\
\mbf{N}(T_A) \ar@{^{(}->}[r] & \mbf{N}(\mc{D}_A).
}
\]
It is thus clear that $\operatorname{LMod}^{\text{perf}}_A$ is categorically equivalent to $\mbf{N}(\operatorname{LMod}^{M,\circ,perf}_A)$. 
\end{proof}

We also have the following

\begin{prop}\label{thick_equiv}
For $A$ a ring spectrum $\mbf{N}(\mc{T}^M_A (S)) \simeq \mc{T}_A (S)$. 
\end{prop}
\begin{proof}
The proof is the same as above. 
\end{proof}

\subsection{Correspondences Between Module Categories}

A key point for us is that for Koszul dual algebras $A$ and $B$ there is a close relationship between their module categories. More specifically, working in derived categories, $\operatorname{Ext}_A (-,S)$ will map $S$ to $B$, and as this preserves colimits will take anything generated by $S$ to an analogous $B$-module. Simlarly, $\operatorname{Ext}_A (A,S) \simeq S$, so $\ext$ will take a compactly generated $A$-module to a corresponding $S$-module. 

The following is a key lemma and this is in some sense the core of what Koszul duality is. 

\begin{lem}\label{koszul}
The categories $\mc{T}^M_A (S)$ and  $\operatorname{LMod}^{M,\operatorname{perf}}_{B^{\text{op}}}$ are contravariantly equivalent as are $\mc{T}^M_{B^{\text{op}}} (S)$ and $\operatorname{LMod}^{M,\operatorname{perf}}_A$. These equivalences are induced by contravariant adjunctions: 
\[
\ext_A( -, S) : \operatorname{LMod}^M_A \leftrightarrows\operatorname{LMod}^M_{B^{\text{op}}} : \ext_{B^{\text{op}}}(-,S)
\]
\end{lem}
\begin{proof}
This follows exactly as in \cite{BlumbergMandellKoszul}. 

\end{proof}

\begin{rmk}
Replacing $A$ by $A^{\text{op}}$ and $B$ by $B^{\text{op}}$ we get the exact same relationship between $\operatorname{RMod}^{\operatorname{perf}}_B$ and $\mc{T}^R_A (S)$. This will be useful for more cleanly phrasing our theorems later. 
\end{rmk}

The comparison statements in the previous section allow us to immediately translate the above result into $\infty$-category language.

\begin{thm}\label{main_koszul}
The $\infty$-categories $\mc{T}_A (S)$ and $\operatorname{LMod}^{\operatorname{perf}}_{B^{\text{op}}}$ are categorically equivalent, as are $\mc{T}_{B^{\text{op}}} (S)$ and $\operatorname{LMod}^{\operatorname{perf}}_A$. 
\end{thm}

\begin{rmk}
Below we offer an alternative path to this statement, hence we do not emphasize the proofs of \ref{koszul} and \ref{main_koszul}. 
\end{rmk}

There is another way to see the equivalence, which will make the presence of ``op''s more apparent. We can use Schwede-Shipley's theory classifying stable categories \cite{schwede_shipley} as reformulated by Lurie \cite{LurieHA}. Before we do so, we need a definition. 

\begin{defn}
The $\infty$-category $\mc{L}_A (S)$ is the localizing subcategory generated by $S$ in $\operatorname{LMod}_A$. 
\end{defn}

The category $\mc{T}_A (S)$ will be stable a $\infty$-category and is the subcategory of compact objects of the localizing subcategory $\mc{L}_A(S)$ generated by $S$.   Furthermore, $\mc{L}_A (S)$ is presentable and is generated by $S$.  Thus, $\mc{L}_A (S)$ is equivalent to $\operatorname{RMod}_R$ for some ring spectrum \cite[Thm. 8.1.2.1]{LurieHA}, and this ring spectrum can be idenified with the endomorphism objects of the generators, i.e. $\operatorname{End}_A (S,S)$ --- but this is precisely $B$ by assumption. We now have

\begin{lem}
$\mc{L}_A (S)^{\operatorname{perf}} \simeq \mc{T}_A (S)$
\end{lem}
\begin{proof}
This is clear. 
\end{proof}

The above lemma tells us that upon taking compact objects on both sides of $\mc{L}_A (S) \simeq \operatorname{RMod}_R$ we get that
\[
\mc{T}_A (S) \simeq \operatorname{RMod}^{\operatorname{perf}}_{B} \simeq \operatorname{LMod}^{\operatorname{perf}}_{B^{\text{op}}}
\]

Thus, we have a slightly more conceptual way of seeing \ref{main_koszul} as well. 

\begin{rmk}
Lurie has also given a definition of Koszul duality \cite{LurieDAGX}. Now that we have set up the necessary machinery, we can observe that his definition can be seen to make contact with this one fairly easily. In \cite[Nota. 3.1.11]{LurieDAGX}, given $A$ and $B$ augmented $S$-algebras and a pairing $A \sma B \to S$ (as above) he defines duality functors 
\begin{align*}
\mf{D} : (\operatorname{LMod}_A)^{\text{op}} &\to \operatorname{LMod}_B\\
\mf{D}': (\operatorname{LMod}_B)^{\text{op}} & \to \operatorname{LMod}_A
\end{align*}
which are represented by $\operatorname{End}_A ( - , S)$ and $\operatorname{End}_B(-, S)$ respectively. These are exactly our functors above. In \cite[Rmk. 3.1.12]{LurieDAGX} he notes that $A$ and $B$ are Koszul dual if and only if $\operatorname{End}_A(S, S) \simeq B$ and $\operatorname{End}_B (S, S) \simeq A$. 
\end{rmk}

Having established some relationship between certain categories of compact modules, one may naturally wonder what kind of compactness conditions $A$ and $B$ themselves must satisfy. The following proposition, which appears in slightly different form in \cite{DGI}, answers this question.

\begin{prop}\cite[Pr.4.17]{DGI}
Let $A \to S$ be an augmentation and $\operatorname{REnd}_R (S, S) \to S$ be an augmentation as well. Then $R \in \mc{T}_R (S)$ if and only if $S \in \mc{T}_S (\operatorname{REnd}_R (S, S))$. 
\end{prop}

\begin{rmk}
A translation of this is that if $R$ is ``finitely built'' from $S$ via fibers, cofibers and retracts, then $S$ must be ``finitely built'' from $\operatorname{REnd}_R (S, S)$ via fibers, cofibers and retracts. 

In particular, if $A$ and $B$ are Koszul dual, then $B \simeq \operatorname{REnd}_B(S, S)$. Thus, $A$ is finitely built from $S$ if and only if $S$ is finitely built from $B$. 
\end{rmk}

Eventually, the above will say that there must be some kind of asymmetry in the dualities we are considering. The dualities we consider have a close connection to subcategories of compact modules, and these will depend heavily on compactness conditions. In the end, the above will cause us to face that fact that our dualities in field theories must necessarily be asymmetric.

\section{Models for stable and symmetric monoidal $(\infty,1)$-categories}

We will have occasion to use a number of models of $(\infty,1)$-categories, symmetric monoidal $(\infty,1)$-categories and stable $(\infty,1)$-categories.  Below we will discuss a convenient model for stable $(\infty,1)$-categories, as well as how to generate symmetric monoidal $(\infty,1)$-categories from more rigid structures. 

\subsection{Some Results on Symmetric Monoidal Relative Categories}

Long ago it was realized by Dwyer and Kan \cite{DK1, DK2, DK3} that categories equipped with weak equivalences were enough to give a ``homotopy theory''. In \cite{BarwickKan} Barwick and Kan rechristen such categories as \textit{relative categories} and show that this provides another model for $(\infty,1)$-categories. Becuase this requires only one auxiliary notion (weak equivalences) and reduces $\infty$-categories to their essence, this will be a useful model for us. Furthermore, there is a rich supply of relative categories coming from forgetting the structure of fibrations and cofibrations in model categories. 

Below, we will be concerned with symmetric monoidal versions of this notion and we can effectively work with these by theorems of Lurie \cite[4.1.3.4]{LurieHA}. 

Before we go on we will give some definitions. 

\begin{defn}
A \textbf{special relative category} is a relative category $(\mbf{C},\mbf{W})$ such that
\begin{enumerate}
\item $\mbf{W}$ contains all objects and isomorphisms in $\mbf{C}$
\item $\mbf{W}$ is a \textit{subcategory} of $\mbf{C}$. 
\end{enumerate}
\end{defn}

\begin{rmk}
In general for a relative category, $\mbf{W}$ need not be a subcategory. 
\end{rmk}

\begin{defn}
A \textbf{symmetric monoidal special relative category} is a special relative category $(\mbf{C},\mbf{W})$ equipped with a symmetric monoidal product $\otimes :\mbf{C}\times \mbf{C} \to \mbf{C}$ such that if $C, D \in \mbf{C}$ and $D \to D'$ is a morphism in $\mbf{W}$ then $C \otimes D \to C \otimes D'$ is in $\mbf{W}$. 
\end{defn}

Now the pair $(\mbf{N}(\mbf{C}), \mbf{N}(W))$ satisfies the conditions of \cite[4.1.4.3]{LurieHA} which allows us to construct a symmetric monoidal $\infty$-category $\mbf{N}(\mbf{C})[\mbf{W}^{-1}]^{\otimes}$ with underlying $\infty$-category $\mbf{N}(\mbf{C})[\mbf{W}^{-1}]$, i.e. the underlying $\infty$-category is $\mbf{N}^{\text{coh}}(\mbf{C}, \mbf{W})$. 

We now have the following theorem which will allow us to lift functors that we define on relative categories to $\infty$-categories. 

\begin{thm}\label{sym_mon}
Suppose that $\mbf{C}^{\otimes}$ and $\mbf{D}^{\otimes}$ are symmetric monoidal special relative categories and that we have a symmetric monoidal functor $F : \mbf{C}^{\otimes} \to \mbf{D}^{\otimes}$ that preserves weak equivalences. Then, we have an induced symmetric monoidal functor on $\infty$-categories:
\[
\mbf{N} \mbf{C}[\mbf{W}^{-1}_{\mbf{C}}]^{\otimes} \to \mbf{N}\mbf{D}[\mbf{W}^{-1}_{\mbf{D}}]^{\otimes}
\]
\end{thm}
\begin{proof}
Given a symmetric monoidal category, $\mbf{C}$, one can construct a colored operad $\mbf{C}^{\otimes}$ \cite[2.0.0.1]{LurieHA} and from there apply the operadic nerve \cite[2.1.1.23] to obtain a symmetric monoidal category $\mbf{N}^{\otimes}(\mbf{C})$. Furthermore, this procedure is functorial, so a symmetric monoidal functor $\mbf{C} \to \mbf{D}$ yields a symmetric monoidal functor of $\infty$-categories $\mbf{N}^{\otimes}(\mbf{C}) \to \mbf{N}^{\otimes}(\mbf{D})$. Now, from \cite[4.1.3.4]{LurieHA} we further have a diagram
\[
\xymatrix{
\mbf{N}^{\otimes}(\mbf{C})\ar[d]\ar[r] & \mbf{N}^{\otimes}(\mbf{D})\\
\mbf{N}\mbf{C}[\mbf{W}^{-1}]^{\otimes}\ar@{.>}[ur] & \\
}
\]
by the universal propery of $\mbf{N}\mbf{C}[\mbf{W}^{-1}_{\mbf{C}}]^{\otimes}$\cite[4.1.3.4 (1)]{LurieHA}. This diagram can be extended further to a diagram 
\[
\xymatrix{
\mbf{N}^{\otimes}(\mbf{C})\ar[d]_L\ar[r]^{\mbf{N}^{\otimes} f} & \mbf{N}^{\otimes}(\mbf{D})\ar[d]^L\\
\mbf{N}\mbf{C}[\mbf{W}^{-1}_{\mbf{C}}]^{\otimes}\ar@{.>}[ur]\ar[r] & \mbf{N} \mbf{D}[\mbf{W}^{-1}_{\mbf{D}}]^{\otimes}
}
\]
simply by the fact that $L$ exists, $\mbf{N}^{\otimes} f$ preserves weak equivalences and $\mbf{N}\mbf{C}[\mbf{W}^{-1}_{\mbf{C}}]$ satisfies a universal property. 
\end{proof}

\subsection{Some Notes on Spectral Categories}

The work of \cite{BGT} demonstrates that the category of small, stable, idempotent complete $\infty$-categories (hereafter referred to as $\mc{C}\text{at}^{\operatorname{perf}}_\infty$) is a natural domain for invariants such as $K$-theory, $\thh$ and $\tc$. We will adopt this point of view later in the paper for $\thh$. However, in order to work with $\mc{C}\text{at}^{\text{perf}}_\infty$ we need a more tractable rigid model for it. The papers \cite{BGT} and \cite{BGT_TC} construct such a rigid model using spectral categories. In this section we review some constructions with spectral categories for later use.

\begin{defn}
A \textbf{spectral category} is a category enriched in spectra.  More explicitly, a spectral category is a category $\mc{D}$ with morphism spectra $\mc{D}(d_0, d_1)$. Furthermore, composition is given by maps of spectra
\[
\mc{D}(d_0,d_1) \sma \mc{D}(d_1, d_2) \to \mc{D}(d_0,d_2)
\]
Since composition must be associative, we must use one of the modern model categories of spectra; typically symmetric spectra are used. 
\end{defn}
\begin{rmk}
The canonical example of such a category is the category of spectra itself. Spectra naturally  have mapping spectra which are composed as above. 
\end{rmk}

Spectral categories come equipped with a model structure where the weak equivalences are DK-equivalences or Morita equivalences (see \cite{BGT} and \cite{TabuadaSpectral}). We refer to the category of spectral categories $\mc{C}\text{at}_{\operatorname{Sp}}$. 

\begin{defn}
Spectral categories come equipped with a monoidal structure. That is, there is an operation $\sma: \mc{C}\text{at}_{\operatorname{Sp}} \times \mc{C}\text{at}_{\operatorname{Sp}} \to \mc{C}\text{at}_{\operatorname{Sp}}$. Given spectral categories $\mc{C}$ and $\mc{D}$ we can form a new spectral category with objects  $\operatorname{ob} \mc{C} \times \operatorname{ob} \mc{D}$ and morphisms $\mc{C} \sma \mc{D} (c_0 \times d_0, c_1 \times d_1) := \mc{C}(c_0, c_1) \sma \mc{D}(d_0, d_1)$. 
\end{defn}

However, the symmetric monoidal structure is not always well behaved with respect to the model category structure. Given two cofibrant spectral categories, their smash product need not be cofibrant. 

 In \cite{BGT}, Blumberg, Gepner and Tabuada get around this difficulty by introducing point-wise cofibrant spectral categories. As one would guess, these are spectral categories all of whose morphism spectra are cofibrant symmetric spectra. Their properties are summarized in Prop 4.2 of \cite{BGT_TC}. As in \cite{BGT_TC} we let $\mc{C}\text{at}^{\operatorname{flat}}_{\Sp}$ denote the full subcategory of pointwise cofibrant spectral categories. We also let $\mc{W}$ denote the Morita equivalences between such categories. 

The above Thm. \ref{sym_mon}  implies that the model category of flat spectral categories forms a symmetric monoidal special relative category. The authors of \cite{BGT_TC} then apply \cite[Prop. 4.1.3.2]{LurieHA} to obtain

\begin{lem}\label{flat}\cite[3.4,3.5]{BGT_TC}
Cofibrant replacement $\mc{C}\text{at}^{\operatorname{flat}}_{\operatorname{Sp}} \to (\mc{C}\text{at}_{\operatorname{Sp}})^c$ induces a categorical equivalence
\[
\mc{C}\text{at}^{\operatorname{perf}}_\infty \simeq \mbf{N}(\mc{C}\text{at}^{\operatorname{flat}}_{\operatorname{Sp}})[\mc{W}^{-1}] \to \mbf{N}((\mc{C}\text{at}_{\operatorname{Sp}})^c)[\mc{W}^{-1}]
\]
\end{lem}

In the sequel, this will allow us to define $\thh$ on spectral categories and pass to $\mc{C}\text{at}^{\operatorname{perf}}_\infty$.

\subsection{Symmetric Monoidal Structures on $\mc{C}\text{at}_\infty$}

One thing that will be extremely important is the product on presentable $\infty$-categories. We recall the theorems from Lurie:

\begin{thm}\cite[Pr.6.3.1.14]{LurieHA}
The category of presentable $\infty$-categories, $\mc{P}\text{r}^L$,  is a symmetric monoidal $\infty$-category. The category of stable, presentable $\infty$-categories, $\mc{P}\text{r}^L_{\text{St}}$, is also a symmetric monoidal $\infty$-category with unit $\operatorname{Sp}$
\end{thm}

The following refinement is found in \cite{BGT} and \cite{BFN}. 

\begin{thm}\label{morita_thm}
The category of small, stable, idempotent complete $\infty$-categories, $\mc{C}\text{at}^{\text{perf}}_\infty$, is a closed symmetric monoidal $\infty$-category. The unit for the product structure is $\operatorname{Sp}^\omega$, the compact spectra, and the internal mapping object is $\fun^{\operatorname{Ex}} (\mc{C},\mc{D})$.
\end{thm}

Idempotent complete $\infty$-categories are essentially categories of compact objects \cite[5.4.2.4]{LurieHTT}, which are naturally the domain of $K$-theory and $\thh$. The above theorem says that there is a tensor product on this category and that it has a unit object. This will enable us to talk about duality in this category. This will be essential for cleanly phrasing Koszul duality. 

As with other $\infty$-categorical constructions we have dealt with, we would like a rigid model for $\mc{C}\text{at}^{\operatorname{perf},\otimes}_\infty$. Luckily, this is provided by \cite{BGT_TC}.

\begin{lem}\label{flat_sym_mon}\cite{BGT_TC}
There is an equivalence of symmetric monoidal $\infty$-categories
\[
(\mc{C}\text{at}^{\operatorname{perf}}_\infty)^{\otimes} \simeq (\mbf{N}(\mc{C}\text{at}^{\operatorname{flat}}_{\operatorname{Sp}})[\mc{W}^{-1}])^{\otimes}
\]
where again $\mc{W}$ are the Morita equivalence. 
\end{lem}

This lemma is essentially saying that $\mc{C}\text{at}^{\text{flat}}_{\Sp}$ is a symmetric monoidal special relative category and that upon passing to $\infty$-categories we get the perfect $\infty$-categories.

The monoidal structure of $\mc{C}\text{at}^{\operatorname{perf}}_\infty$ will play an important role in the sequel. We now discuss the Morita theoretic characterization of the monoidal struture on $\cat^{\operatorname{perf}}_\infty$. The theorem is in \cite{BGT}, but at points it will be easier to work with the guts of the theorem, so we need some introduction. 

\begin{rmk}
We will have occasional to use Ind-categories \cite[Ch. 5]{LurieHTT} below. We do not have space to discuss these in depth (nor can we improve on Lurie's exposition), but we can indicate definitions. Given an $\infty$-category $\mc{C}$, $\operatorname{Ind}(\mc{C})$ will be the category we get by formally adjoining filtered colimits. Since filtered colimits play a role in defining compact objects, it is not so unexpected that $\operatorname{Ind}(\mc{C})$ categories will play a role in questions involving compact objects. In particular, if we take $\operatorname{Ind}(\mc{C})$ and take the compact objects of that category, we obtain a kind of ``completion'' of $\mc{C}$, the idempotent completion \cite{LurieHTT}. The notion of $\operatorname{Ind}$-categories also allows us to define compactly-generated categories. These are categories $\mc{C}$ such that there are subcategories $\mc{C}^0 \subset \mc{C}$ such that the object of $\mc{C}^0$ are compact and $\operatorname{Ind}(\mc{C}^0) \simeq \mc{C}$ (again, see \cite[Ch. 5]{LurieHTT} for further explanation). This will be used in \ref{thick_equivalence}.

\end{rmk}

Arguably one of the more important parts of this monoidal structure is the existence of various maps from the unit object. Specifically, for any $c \in \mc{C}$ where $\mc{C}$ is an object of $\cat^{\operatorname{perf}}_\infty$, there is a corresponding map from spectra to $\mc{C}$, which we define as follows. 

\begin{defn}\label{mu}
Let $\mc{C}$ be a stable $\infty$-category and let $c \in \mc{C}$ be an object of that category. We define a functor $\mu_c : \operatorname{Sp} \to \mc{C}$ by specifying that the functor be colimit preserving and mapping $S$ to $c$. 
\end{defn}

\begin{defn}
For $\mc{A}$ a stable $\infty$-category, there is a map
\[
\Sp \to \mc{A} \xrightarrow{j} \operatorname{Ind}(\mc{A})
\]
which upon taking compact objects becomes
\[
\operatorname{Sp}^\omega \to \mc{A}^\omega \to \operatorname{Ind}(\mc{A})^\omega \simeq \operatorname{Idem}(\mc{A}).
\]
In particular, for a stable, idempotent complete $\infty$-category, $\operatorname{Idem}(\mc{A}) \simeq \mc{A}$, so $j \circ \mu_c: \operatorname{Sp} \to \mc{A} \to \operatorname{Ind}(\mc{A})$ becomes a map $\mu^\omega_c : \operatorname{Sp}^\omega \to \mc{A}$. 

Finally, we tensor with $\mc{B}^{\text{op}}$ and map into $\operatorname{Sp}$ to get the following map
\[
\nu_a : \fun^{\operatorname{Ex}} (\mc{A}\widehat{\otimes} \mc{B}^{\text{op}}, \operatorname{Sp}) \to \fun^{\operatorname{Ex}}(\mc{B}^{\text{op}}, \operatorname{Sp}) \simeq \mc{B}
\]
This last equivalence is by \cite[Pr. 6.3.1.16]{LurieHA} and the fact that $\mc{B}$ is stable so that $\mc{B}\otimes \operatorname{Sp} \simeq \mc{B}$. 
\end{defn}
Having defined $\nu_a$, we define the notion of right-compactness. 
\begin{defn}
If for an element $F \in \fun^{\text{Ex}} (\mc{A}\widehat{\otimes} \mc{B}^{\text{op}}, \Sp)$, the image of $\nu_a$ is compact for \textit{every} $a \in \mc{A}$, then the element is \textbf{right-compact}. We will denote full subcategory of right-compact functors by $\fun^{\operatorname{RC}}(\mc{A},\mc{B})$. 
\end{defn}

The following appears as \cite[Cor. 2.16]{BGT}:

\begin{cor}
For $\mc{A}, \mc{B} \in \operatorname{Cat}^{\operatorname{perf}}_\infty$ there is an equivalence of $\infty$-categories
\[
\fun^{\operatorname{Ex}}(\mc{A},\mc{B}) \simeq \fun^{\operatorname{Ex,RC}}(\mc{A}\widehat{\otimes}\mc{B}^{\text{op}}, \Sp).
\]
\end{cor}

\begin{example}\label{dual_cat}
We give the most important example of the process above: we compute the \textit{left dual} (\ref{dualizable}) of a small, stable, idempotent complete category, $\mc{C} \in \cat^{\operatorname{perf}}_\infty$. The dual of this is computed as $\fun^{\operatorname{Ex}}(\mc{C}, \operatorname{Sp}^\omega)$, but we would like a more explicit description. There is a map
\[
\fun^{\operatorname{Ex}}(\mc{C}, \Sp^\omega) \to \fun^{\operatorname{Ex}}(\mc{C}, \fun^{\operatorname{Ex}}(\Sp^{\omega, \text{op}}, \Sp)) \simeq \fun^{\operatorname{Ex}}(\mc{C}, \Sp)
\]
where the final equivalence follows from the fact that 
\begin{align*}
\fun^{\operatorname{Ex}}(\Sp^{\omega, \text{op}}, \Sp) &\simeq \fun^L (\operatorname{Ind}(\Sp^{\omega, \text{op}}), \Sp) \simeq \fun^L(\Sp^\text{op}, \Sp) \\
&\simeq \fun^R (\Sp, \Sp^{\text{op}})^{\text{op}} \simeq \Sp
\end{align*}

Furthermore, 
\[
\fun^{\operatorname{Ex}}(\mc{C}, \Sp) \simeq \fun^L(\operatorname{Ind}(\mc{C}), \Sp) \simeq \operatorname{Ind}(\mc{C})^{\text{op}}
\]

Let $(\mu^\omega_c)^\ast : \operatorname{Ind}(\mc{C})^{\text{op}} \to \operatorname{Sp}$ be the map induced by $\mu^\omega_c : \operatorname{Sp}^\omega \to \mc{C}$ via
\[
\operatorname{Ind}(\mc{C})^{\text{op}} \simeq \fun^{\text{Ex}}(\mc{C}, \operatorname{Sp}) \to \fun^{\operatorname{Ex}}(\operatorname{Sp}^\omega, \Sp) \simeq \operatorname{Sp}
\]
 As above, we get induced maps 
\[
\nu_c : \fun^{\operatorname{Ex}}(\mc{C}, \Sp^\omega) \to \fun^{\operatorname{Ex}}(\mc{C}, \Sp) \xrightarrow{(\mu^\omega_c)^\ast}  \fun^{\operatorname{Ex}}(\Sp^{\omega}, \Sp) \simeq \Sp
\]

Then, then $\fun^{\operatorname{Ex}}(\mc{C}, \Sp^\omega)$ corresponds exactly to those elements of $\operatorname{Ind}(\mc{C})^{\text{op}}$ that map to compact spectra under \textit{all} maps $\nu_c$. That is, for a particular $c \in \mc{C}$, we have a pullback along inclusions of $\infty$-categories
\[
\xymatrix{
\mc{E}_c \ar[d] \ar@{^{(}->}[r] & \operatorname{Ind}(\mc{C})^{\text{op}} \ar[d]^{(\mu^\omega_c)^\ast}\\
\Sp^\omega \ar@{^{(}->}[r] & \Sp 
}
\]
and 
\[
\fun^{\operatorname{Ex}} (\mc{C}, \Sp^\omega) \simeq \bigcap_{c \in \mc{C}} \mc{E}_c
\]
where the intersection is taken inside of $\operatorname{Ind}(\mc{C})^{\text{op}}$. 

That is, $\fun^{\text{Ex}}(\mc{C}, \Sp^\omega)$ are exactly those objects that when considered as spectra via $\mu^\omega_c$ are compact. 
\end{example}

\begin{rmk}
There is a subtlety here which confused the author for some time. A cursory glance would make it appear that the dual of a category is in fact just the \textit{opposite} category. Indeed, this would be true for presentable $\infty$-categories with their symmetric monoidal product in the theorem above \cite{LurieHA}, \cite{BGT}. However, the presence of right-compactness alters this. 
\end{rmk}

\section{The Theorem}

The main step to the proof of the duality theorem is proving a duality between module categories of Koszul dual algebras. However, there are finicky technical details to pin down and we must be somewhat careful. In particular, since we are not (necessarily) dealing with $\mbf{E}_\infty$-rings we have to distinguish between left and right modules. This in fact ends up giving a kind of asymmetry in specific examples where we merely have associativity and no higher coherence. 

Once we have produced a duality between module categories, we need to demonstrate that $\thh$ is a symmetric monoidal functor from the $\infty$-category of perfect stable $\infty$-categories to the $\infty$-category of spectra.  It is shown to be symmetric monoidal in \cite{BGT_TC}.  

Finally, we need to at least state a Morita invariance result that guarantees what we normally think of as topological Hochschild homology of a ring is the same as taking topological Hochschild homology of a particular symmetric monoidal $\infty$-category. 

\subsection{Duality of Module Categories}

We have been driving home the point that the Koszul duality is really encoded as a relationship between underlying module categories, at least for stable categories. We have previously stated this as a relationship between certain thick subcategories of modules. Below we prove that the categories of modules are in some sense dual in the $\infty$-category of perfect stable $\infty$-categories.

The following illustrates why we must be careful: our way forward is not as easy as we might hope.  

\begin{example}
Our canonical example is slightly confusing. For the ring spectrum $A = \Sigma^\infty_+ \Omega X$ and $B = DX$, the latter is $\mbf{E}_\infty$ but the former is only $\mbf{E}_1$.  Thus, $\operatorname{LMod}_B \simeq \operatorname{RMod}_B$, but $\operatorname{LMod}_A$ is \textit{not} necessarily equivalent to $\operatorname{RMod}_A$.  This goes against the expectation that the Koszul dual of something that is $\mbf{E}_1$ would also be (exactly) $\mbf{E}_1$. Furthermore, there is an anti-involution on $\Sigma^\infty_+ \Omega X$ which can be used to turn left modules into right modules. This further complicates issues. 

 Examples such as this lead us to be especially careful when considering left and right modules and opposite categories thereof. 
\end{example}

We also need to be careful about what we mean by ``duality'' since our dualities will have handedness; in general, there is no reason to expect $\thh(\mc{C})$ to be a dualizable object in $\operatorname{Sp}$. When one thinks of duality, often one thinks of mapping an object into some unit object (e.g. taking duals of vector spaces). However, it is more productive to think abstractly in terms of unit and counit maps. This will allow us to distinguish the handedness of dualities. 

Below we state the definition of duality in $\infty$-categories \cite{LurieHA}. Note that a symmetric monoidal $\infty$-category has a homotopy category, $\operatorname{Ho}(\mc{C})$ that is symmetric monoidal in the usual sense, and that duality in $\mc{C}$ is determined by the homotopy category. 

\begin{defn}\label{dualizable}
Let $\mc{C}^{\otimes}$ be a symmetric monoidal $\infty$-category with underlying $\infty$-category $\mc{C}^{\otimes}_{\langle 1 \rangle} \simeq \mc{C}$ and unit $U \in \mc{C}$. Let $C \in \mc{C}$ be an object. Then $C$ is \textbf{left dualizable} if there is a $C^\vee \in \mc{C}$ and maps $\tau : C \otimes C^\vee \to U$ and $\chi : U \to C^\vee \otimes C$ such that
\[
C \xrightarrow{\id \otimes \chi} C \otimes C^\vee \otimes C \xrightarrow{\tau \otimes \id}  C
\]
is the identity in $\operatorname{Ho}(\mc{C})$. 

Similarly $C$ is \textbf{right dualizable} if
\[
C^\vee \xrightarrow{\chi \otimes \id}  C^\vee \otimes C\otimes C^\vee \xrightarrow{\id \times \tau}  C^\vee
\]
is an equivalence in $\operatorname{Ho}(\mc{C})$. 

We call $C$ \textbf{dualizable} if it is both left and right dualizable. 
\end{defn}

\begin{rmk}
In a symmetric monoidal category with internal function objects \cite[Def. 4.2.1.28]{LurieHA}, we can easily write down a left dual. Let $A$ be an object in such a category and $\mbf{1}$ the unit for the symmetric monoidal structure. The dual is $\operatorname{Map}(A, \mbf{1})$. The map $\tau: \operatorname{Map}(A, \mbf{1}) \otimes A \to \mbf{1}$ is given by evaluation, and $\chi: \mbf{1} \to \operatorname{Map}(A, \mbf{1}) \otimes A$ is given by the adjoint to $\operatorname{Map}$. We often use the notation $DA$ for the mapping space from $A$ into the unit object. 
\end{rmk}

\begin{rmk}
Because of various assymetries that arise, the concept of dualizability will not play a large role in the sequel. There is a concrete characterization of dualizable objects of $\mc{C}\text{at}^{\text{perf}}_\infty$ in \cite{BGT}. 
\end{rmk}

We now apply Ex.~\ref{dual_cat} to $\operatorname{LMod}^{\operatorname{perf}}_A$ to find the dual. 

\begin{prop}
The left dual of $\operatorname{LMod}^{\operatorname{perf}}_A$ in $\cat^{\operatorname{perf}}_\infty$ is $\operatorname{LMod}^{S\text{-comp}}_A$, the $A$-modules that are compact when considered as $S$-modules.  
\end{prop}

\begin{proof}
The functors $\mu^\omega_M : \Sp^\omega \to \operatorname{LMod}^{\operatorname{perf}}_A$ are the unique colimit preserving functors that take the sphere spectrum $S$ to an $A$-module $M$. We get, as in the example, a pullback along inclusions
\[
\xymatrix{
\mc{E}_M \ar[d]\ar@{^{(}->}[r] & \operatorname{LMod}^{\text{op}}_A \ar[d]^{M=(\mu^\omega_M)^\ast} \\
\Sp^\omega \ar@{^{(}->}[r] & \Sp
}
\]
and
\[
\fun^{\operatorname{Ex}}(\operatorname{LMod}^{\text{op}}_A, \Sp^\omega) \simeq \bigcap_M \mc{E}_M
\]
Before we say anything concrete, we need to know the nature of the map $M: \operatorname{LMod}^{\text{op}}_A \to \operatorname{Sp}$. As in Ex.~\ref{dual_cat}, the map comes from 
\[
\operatorname{LMod}^{\text{op}}_A \simeq \operatorname{Ind}(\operatorname{LMod}^{\text{perf}}_A)^{\text{op}} \simeq \fun^{\operatorname{Ex}}(\operatorname{LMod}^{\operatorname{perf}}_A, \Sp) \xrightarrow{(\mu_M)^\ast} \fun^{\operatorname{Ex}} (\Sp^\omega, \Sp) \simeq \Sp
\]
That is, we are considering an $A$-modules as an $S$-module via
\[
\Sp^\omega \xrightarrow{M} \operatorname{LMod}^{\text{op}}_A \to \Sp.
\]
The $S$-module resulting from the above pullback will be perfect exactly when $M \in \operatorname{LMod}_A$ is perfect as an $S$-module.  That is
\[
\mc{E}_M \simeq \operatorname{LMod}^{S-\text{perf}}_A
\]
and so
\[
\fun^{\text{Ex}}(\operatorname{LMod}^{\text{op}}_A, \Sp^\omega) \simeq \bigcap_M \operatorname{LMod}^{S-\text{perf}}_A \simeq \operatorname{LMod}^{S-\text{perf}}_A. 
\]
This completes the proof. 
\end{proof}

We will need the following result to induce equivalences between module categories:

\begin{lem}\label{thick_and_S_comp}
Let $A$ be an augmented $\mbf{E}_1$-spectrum such that $A$ is small as an $S$-module. As above let $\mc{T}_A (S)$ denote the thick subcategory generated by $S$ in $A$. Let $\operatorname{LMod}^{S\text{-comp}}_A$ denote the left $A$-modules that are compact as $S$-modules. Then
\[
\mc{T}_A (S) \simeq \operatorname{LMod}^{S\text{-comp}}_A
\]
\end{lem}

This actually follows from a slightly more general proposition. 

\begin{prop}\label{thick_equivalence}
Let $\mc{C}$ be a small, stable, idempotent complete $\infty$-category and let $\mc{D}$ be an $\infty$-category generated by $\mc{C}$, i.e. $\mc{D} \simeq \operatorname{Ind}(\mc{C})$. Suppose $\mc{C}' \in \mc{C}\text{at}^{\text{perf}}_\infty$ sits in the following string of fully faithful inclusions
$\mc{C} \subset \mc{C}' \subset \mc{D}$.
Then $\mc{C}$ is categorically equivalent to $\mc{C}'$. 
\end{prop}
\begin{proof}
Note that $\operatorname{Ind}$ is functorial. Apply it to the sequence of inclusions above to get
\[
\operatorname{Ind}(\mc{C}) \to \operatorname{Ind}(\mc{C}') \to \operatorname{Ind}(\mc{D}).
\]
Since $\mc{D}$ is compactly generated we have $\operatorname{Ind}(\mc{D}) \simeq \mc{D}$. Furthermore, since $\mc{C}$ generates $\mc{D}$, $\operatorname{Ind}(\mc{C}) \simeq \mc{D}$ and we are left with maps
\[
\mc{D} \to \operatorname{Ind}(\mc{C}') \to \mc{D}. 
\]
all of which are fully faithful \cite[5.3.5.11]{LurieHTT}. It is easy to see that all of these maps are essentially surjective as well. Thus, $\operatorname{Ind}(\mc{C}) \simeq \operatorname{Ind}(\mc{C}') \simeq \operatorname{Ind}(\mc{D})$. Upon taking compact objects of these categories we obtain
\[
\operatorname{Ind}(\mc{C})^\omega \simeq \operatorname{Ind}(\mc{C}')^\omega \simeq \operatorname{Ind}(\mc{D})^\omega.
\]

However, $\operatorname{Ind}(\mc{C})^\omega$ models the idempotent completion of $\mc{C}$ \cite[Lem. 5.4.2.4]{LurieHTT}, and since $\mc{C}$ and $\mc{C}'$ are already idempotent complete we get $\mc{C} \simeq \mc{C}'$. 
\end{proof}

\begin{rmk}
In light of the Morita theory of \cite{BGT}, this proposition is more or less saying that all spectral categories that are contained in the perfect $A$-modules are Morita equivalent. 
\end{rmk}

We can now prove the lemma above:

\begin{proof}[Proof of Lemma]
In the setup of the lemma we have inclusions
\[
\operatorname{LMod}^{\operatorname{perf}}_A \subset \mc{T}_A (S) \subset \operatorname{LMod}^{S\text{-comp}}_A \subset \operatorname{LMod}_A
\]
We have to see that these satisfy the hypotheses of the proposition above. First, by \cite[Pr. 8.2.5.2]{LurieHA} $\operatorname{LMod}^{\operatorname{perf}}_A \simeq \operatorname{LMod}^\omega_A$ and $\operatorname{Ind}(\operatorname{LMod}^{\operatorname{perf}}_A) \to \operatorname{LMod}_A$ is a categorical equivalence. The inclusions are fully faithful by definition, being inclusions of $\infty$-subcategories. Thus, we may apply Prop. ~\ref{thick_equivalence}. Furthermore, the two middle categories are idempotent complete, so we are done.  
\end{proof}

We now have:

\begin{prop}\label{koszul_duality_categories}
In $\mc{C}\text{at}^{\operatorname{perf}}_\infty$ with the symmetric monoidal structure described above, $\operatorname{LMod}^{\operatorname{perf}}_A$ is left dual to the subcategory of $\operatorname{LMod}_A$ modules that are compact when considered as $S$-modules. Thus, when $A$ is compact as an $S$-module 
\begin{align*}
D(\operatorname{LMod}^{\operatorname{perf}}_A)  \simeq \mc{T}_{A}(S)  \simeq \operatorname{LMod}^{\operatorname{perf}}_{B^{\text{op}}}
\end{align*}
\end{prop}
\begin{proof}
We consider the computation of the dual in \ref{dual_cat} in the case where $\mc{C}$ is a module category. First, we note that we have a canonical functor $\operatorname{Sp} \to \operatorname{LMod}_A$

We get
\[
D(\operatorname{LMod}^{\operatorname{perf}}_A) \simeq \operatorname{LMod}^{S\text{-comp}}_A
\]
Prop. \ref{thick_and_S_comp} above gives
\[
\operatorname{LMod}^{S\text{-comp}}_A \simeq \mc{T}_A (S)
\]
and finally, by our assumption that $A$ and $B$ are Koszul dual, along with our computation in \ref{main_koszul} we have
\[
\mc{T}_A (S) \simeq \operatorname{LMod}^{\operatorname{perf}}_{B^{\text{op}}}. 
\]
This completes the proof. 
\end{proof}

\subsection{Topological Hochschild Homology}

We recall the definition of topological hochschild homology for a spectral category \cite{BGT_TC}. 

\begin{defn}
Let $\mc{C}$ be a spectral category. Then we define $\thh(\mc{C})$ to be the geometric realization of a simplicial object built out of mapping spectra:
\[
\thh(\mc{C}) = \left| \bigvee_{c_i \in \mc{C}} \mc{C}(c_0, c_1) \sma \cdots \sma \mc{C}(c_n,c_0) \right|
\]
\end{defn}

We may make this definition, however, it only has the correct homotopy type when smash product preserves weak equivalences. That is, we have to work in pointwise-cofibrant or \textit{flat} spectral categories, as above. If we prove our theorems for flat spectral categories, then Lem.~\ref{flat} and Lem.~\ref{flat_sym_mon} will allow us to descend to theorems about $\mc{C}\text{at}^{\text{perf}}_\infty$. 

We will need the following theorem. 

\begin{thm}\label{thh_sym_mon_spectral}
$\thh : (\mc{C}\text{at}^{\operatorname{flat}}_{\operatorname{Sp}},\sma) \to (\Sp,\sma)$ is a symmetric monoidal functor of relative categories.  
\end{thm}

\begin{proof}
As indicated in \cite{BGT_TC} we use the shuffle product. The shuffle maps are
\begin{equation}\label{shuffle}
\thh(\mc{C}_1) \sma \cdots \sma  \thh(\mc{C}_n) \to \thh(\mc{C}_1 \sma \cdots \sma \mc{C}_n)
\end{equation}
where we consider the left as a realization (diagonal) of an $n$-simplicial object in spectra. The maps on $k$-simplices look like
\begin{align*}
&\bigvee_{c_{i,1} \in \mc{C}_1} \mc{C}_1(c_{0,1},c_{1,1}) \sma \cdots \sma \mc{C}(c_{k,1}, c_{0,1}) \sma \bigvee_{c_{i,2} \in \mc{C}_2} \mc{C}_1(c_{0,2},c_{1,2}) \sma \cdots \sma \mc{C}(c_{k,2}, c_{0,2})\\
&\cdots \sma \bigvee_{c_{i,n} \in \mc{C}_n} \mc{C}_1(c_{0,n},c_{1,n}) \sma \cdots \sma \mc{C}(c_{k,n}, c_{0,n})\\
&\to \bigvee (\mc{C}_1 \sma \cdots \sma \mc{C}_n)(c_{0,1} \sma \cdots \sma c_{0,n},c_{1,1} \sma \cdots \sma c_{1,n}) \sma \cdots 
\end{align*}
where the maps are given by the fact that we have maps
\[
\mc{C}(c_0,c_1) \sma \mc{D}(d_0,d_1) \cong \mc{C}\sma \mc{D}(c_0 \sma d_0, c_1 \sma d_1)
\]
for spectral categories. 

The result \cite[X.1.3.(iv)]{EKMM} then gives that Eq. ~ \ref{shuffle} is an equivalence. 

\end{proof}

\begin{rmk}
It is more fun to use the siftedness of $\mbf{N}(\Delta^{\operatorname{op}})$ as an $\infty$-category to see that the equivalence of shuffle maps. Recall that simplicial object in a model category determines a simplicial object in an $\infty$-category via taking coherent nerves, with $\Delta^{\text{op}}$ considered as a discrete simplicial category. This also works for multisimplicial objects. Recall also that a simplicial set $K$ is sifted if $K \to K \times K$ is cofinal \cite[5.5.8.1]{LurieHTT}. 

By the above remarks, we obtain the following diagram
\[
\xymatrix{
\mbf{N}(\Delta^{\text{op}}) \ar[d]\ar[r] & \mbf{N}(\Delta^{\text{op}} \times \Delta^{\text{op}}) \ar[d]\ar[r] & \operatorname{Sp} \times \operatorname{Sp} \ar[r]^{\sma} & \operatorname{Sp}\\
 \ast\ar[urrr] & \ast\ar@{.>}[ur] \ar@{.>}[urr] & & \\ 
}
\]

Each of the diagonal arrows is a left Kan extension to a point (i.e. a colimit, which in the contex of $\infty$-categories is geometric realization). Since $\sma$ is designed to commute with colimits separately in each variable, the two dotted arrows map to the same object up to homotopy, and since $\mbf{N}(\Delta^{\text{op}})$ is sifted, the filled diagonal arrow maps to  objects. This is the statement that
\[
\diag |K_\ast \sma L_\ast | \simeq |K_\ast| \sma |L_\ast|
\]
\end{rmk}

Now, as a consequence of the theorem for relative categories, we have the following theorem, also proved in \cite{BGT_TC}. 

\begin{thm}\label{thh_sym_mon_infcat}
$\thh : \cat^{\operatorname{perf}}_\infty \to \Sp$ is a symmetric monoidal functor of $\infty$-categories. 
\end{thm}
\begin{proof}
Combine Thm. ~\ref{thh_sym_mon_spectral}, Lem. 3.9 and Thm. 3.3
\end{proof}

We now have to say a little about Morita invariance. Classically, Morita invariance for Hochschild homology states that if $A$ and $B$ are Morita equivalent, then $\operatorname{HH}_\ast (A) = \operatorname{HH}_\ast (B)$, i.e. Hochschild homology depends only on the underlying category of modules. Similar statements for spectral categories are given by Blumberg and Mandell in \cite{BlumbergMandellWaldhausen}. Here is the statement we will need:

\begin{lem}\cite{BGT}\label{morita_equiv}
Let $A$ be a ring spectrum. Then we have an equivalence
\[
\thh(A) \simeq \thh(\operatorname{LMod}^{\text{perf}}_A).
\]
\end{lem}

\subsection{The Main Theorem and Application}

We apply what we have done above to prove our main result. 

First, we recall that the fact that $\mc{C}\text{at}^{\text{perf}}_\infty$ is symmetric monoidal with internal function objects gives us coevaluation maps
\[
\Sp^\omega \to \mc{C} \otimes \fun^{\operatorname{Ex}} (\mc{C}, \Sp^\omega). 
\]
adjoint to the identity map.

\begin{lem}\label{thh_left_dual}
Let $\mc{C}$ and $\mc{D}$ be categories which are (left) dual in $\mc{C}\text{at}^{\operatorname{perf}}_\infty$. Then
\[
\thh(\mc{C}) \simeq D(\thh(\mc{D}))
\]
\end{lem}
\begin{proof}
This is simply the fact that $\thh$ is symmetric monoidal (Thm ~\ref{thh_sym_mon_infcat}) and the fact that $\mc{C}$ and $\mc{D}$ are left dual may be expressed purely in terms of evaluation and coevaluation maps. That is, the map
\[
\mc{C} \to \mc{C} \otimes \mc{D} \otimes \mc{C} \to \mc{C}
\]
must be the identity. This will be preserved by any symmetric monoidal functor. 

\end{proof}

Finally, we may state and prove the main theorem of the introduction. Note that by examples the compactness statements, as well as carefully keeping track of opposite algebras is essential. There is no reason the theorem should be true (at least not from module category considerations) under other assumptions. 

\begin{thm}\label{main_thm}
Let $A$ and $B$ be $\mbf{E}_1$-Koszul dual ring spectra, and further assume that $A$ is finitely built as an $S$-module. Then, we have
\[
D(\thh(A)) \simeq \thh(B^{\text{op}})
\]
\end{thm}
\begin{proof}
By the Morita equivalence of \ref{morita_equiv}, we have $\thh(A) \simeq \thh(\operatorname{LMod}^{\text{perf}}_A)$. Furthermore, by Prop. ~\ref{koszul_duality_categories} we have $D(\operatorname{LMod}^{\text{perf}}_A) \simeq \operatorname{LMod}^{\text{perf}}_B$. This exhibits $\operatorname{LMod}^{\operatorname{perf}}_{B^{\text{op}}}$ as left-dual to $\operatorname{LMod}^{\operatorname{perf}}_A$. Invoking Lem. ~\ref{thh_left_dual} we finally have the chain of equivalences
\[
D(\thh(A)) \simeq D(\thh(\operatorname{LMod}^{\operatorname{perf}}_A)) \simeq \thh(\operatorname{LMod}^{\operatorname{perf}}_{B^{\text{op}}}) \simeq \thh(B^{\text{op}}).
\]
This completes the proof. 
\end{proof}

Using the theorem and work of \cite{BlumbergMandellKoszul} we  recover Ralph Cohen's original observation:

\begin{cor}
Let $X$ be a simply connected, finite CW-complex, then
\[
 D(\thh (DX))\simeq \thh(\Sigma^\infty_+ \Omega X) 
\]
\end{cor}

\begin{rmk}
Note that in general there is an asymmetry so that $D(\thh(\Sigma^\infty_+\Omega X)) \neq \thh(DX)$. In particular, this prevents us from trying to find the Spanier-Whitehead dual of $\Sigma^\infty_+ \mc{L} X$. 
\end{rmk}

\begin{rmk}
The right hand side is quite easy to compute, so provides another way to see B\"{o}kstedt-Waldhausen's classical computation that
\[
\thh(\Sigma^\infty_+ \Omega X) \simeq \Sigma^\infty_+ \mc{L} X. 
\]
\end{rmk}

\section{Extensions and Further Directions}

As mentioned in the introduction, there are a few other directions this work could be generalized in. In this paper we have been working over $S$ as a ``ground field''. There is reason to believe that finding Koszul duals augmented over $S$ may be harder than finding Koszul duals augmented over some other ground ring $R$. For example, examples of Koszul duality over other ground rings appear in a paper of Baker and Lazarev \cite{baker_lazarev}. It would thus be useful to know the main theorem of the paper for $\thh_R$. 

To this end, the following theorem will appear in future work of the author. 

\begin{thm}\label{thm_over_R}
$\thh_R: \mc{C}\text{at}^{\text{perf}}_{\infty, R} \to \Sp$ is a symmetric monoidal functor. Furthermore for an augmented $R$-algebra $A$, and its Koszul dual $B$, we have
\[
\thh_R (A) \simeq D_R (\thh_R (B))
\]
\end{thm}

\begin{rmk}
This theorem follows immediately, once the necessary foundations are in place. Once we know that categories enriched in $R$-modules can be modeled by spectral categories enriched in $R$-modules, known properties of $\thh$ (Morita invariance, symmetric monoidality) force the theorem. 
\end{rmk}

This provides a duality over other ground rings, and may lead to ways to compute $\thh_R$. For example, we have the following corollaries from the examples in \cite{baker_lazarev}.

\begin{cor}
$\operatorname{MU}^\wedge_p$ and $B_p = F_{MU^\wedge_p}(H\mbf{F}_p, H\mbf{F}_p)$ are Koszul dual. Thus
\[
\thh_{\operatorname{H}\mbf{F}_p} (\operatorname{MU}^\wedge_p) \simeq D_{H\mbf{F}_p} (\thh_{H\mbf{F}_p} (B_p))
\]
\end{cor}

Another result that is a consequence of \ref{thm_over_R} is Jones and McCleary's result \cite{jones_mccleary} which served as inspiration for the present result. 

\begin{thm}[\cite{jones_mccleary}]
Let $X$ be a simply-connected space and $k$ a field. Then 
\[
\operatorname{HC}_\ast (C_\ast (\Omega X); k) \cong \hom(\operatorname{HC}_\ast (C^\ast (X;k)), k)
\]
where $\operatorname{HC}_\ast$ denotes the Hochschild chains. 

\end{thm}
\begin{proof}
By \cite{Shipley}, $k$-DGAs are equivalent to $Hk$-module spectra. The $Hk$-module spectrum that corresponds to $C_\ast (\Omega X; k)$ is $\Sigma^\infty_+ \Omega X \sma Hk$ and the $Hk$-module spectrum that corresponds to $C^\ast (X; k)$ is $DX \sma Hk$. By computations of Adams \cite{Adams}, $C_\ast (\Omega X ;k)$ and $C^\ast (X;k)$ are Koszul dual. Thus, $\Sigma^\infty_+ \Omega X \sma Hk$ is Koszul dual over $Hk$ to $DX \sma Hk$. By \ref{thm_over_R} with $R = Hk$, $A = \Sigma^\infty_+ \Omega X \sma Hk$ and $B = DX \sma Hk$ we are now done. 
\end{proof}

As in \cite{jones_mccleary}, this provides an easy computation of $\hh_\ast (C_\ast (\Omega X; k))$

\begin{cor}
Let $X$ be a simply connected space, $k$ a field, and $\mc{L} X = \operatorname{Map}(S^1, X)$ the free loop space. Then we have the isomorphism
\[
\hh_\ast (C_\ast (\Omega X; k)) \cong H_\ast (\mc{L} X; k)
\]
\end{cor}

There should be many other examples of Koszul duality over ring spectra other than $S$. Such examples might provide interesting computational results, given Thm.  \ref{thm_over_R}.

Another question we may ask is the following: Given that we have asked questions about $\thh$ and $K$-theory and Koszul duality, we could also wonder how Koszul duality behaves for topological cyclic homology, $\tc$. The author does not know the answer to this question at present --- a complication is that we would have to content with a cyclotomic structure and a kind of ``cocyclotomic'' structure.

\section{Relationship with Field Theories and Unbridled Speculation}

\subsection{Field Theories}

Here we briefly outline how the above relates to, and was motivated by field theories. Like Koszul duality, ``field theory'' means many different things to different people, but here  we will mean a field theory in the sense of Lurie \cite{LurieFT}, i.e. a symmetric monoidal functor of $(\infty,n)$-categories $\mbf{Cob}^\amalg \to \mc{C}^{\otimes}$. 

We first recall the definition, due to Atiyah, of a topological field theory. We must first define the domain of a topological field theory. 

\begin{defn}
The \textbf{cobordism category} of $n$-dimensional manifolds is a category whose objects are oriented $n$-manifolds and whose morphisms are cobordisms between $n$-manifolds. The category is symmetric monoidal with product given by disjoint union. Let $\mbf{Cob}^{\amalg}$ denote this category --- the decoration reminds us of the symmetric monoidal structure. 
\end{defn}

We can now define a topological field theory:

\begin{defn}
Let $\mc{C}^{\otimes}$ be a symmetric monoidal category. A \textbf{topological field theory} is a symmetric monoidal functor $F: \mbf{Cob}^{\amalg} \to \mc{C}^{\otimes}$. 
\end{defn}

\begin{rmk}
A topological field theory is thus a functor that assigns to each manifold some algebraic invariant that must in some way respect cobordism and symmetric monoidal structure. 
\end{rmk}

Before we move on it, it would be good to collect some motivating examples. 

\begin{example}
We consider the example a one-dimensional field theory, i.e. a functor $F$ from the cobordism category, valued in vector spaces. This is a functor that assigns to a (positively oriented) point, some vector space $F(+) = V$ (with ground field, say, $k$). It assigns to the oppositely oriented point the dual of the vector space $F(-) = V^{\vee}$. In the bordism category, the circle may be viewed as a morphism between the empty set and the empty set which is a composition of $\emptyset \to + \amalg -$ and $+ \amalg - \to \emptyset$.  The invariant assigned to a circle is thus the composition of the counit and the unit $F(\emptyset) = k \to V\otimes V^{\vee} \to k$. This is a version of the trace; see \cite{LurieFT} for more details. 
\end{example}

\begin{example}
A slightly less well known example is that of the category of bimodules. Let $\mbf{BiMod}$ be the category where the objects are rings and the morphisms are bimodules. The symmetric monoidal structure is given by relative tensor product. That is, if we have ``morphisms'' $ _{A} M_{B}$ and $ _{B} N_{C}$, then their composition is $_A M \otimes_B N_C$. In this case, the invariant associated to a line from a point to a point is the bimodule $_A A_A$. The invariant associated to the bordism from the empty set to two points is the bimodule $A_{A \otimes A^{\text{op}}}$. Similarly, the invariant associated to the bordism from two points to the empty set is $_{A\otimes A^{\text{op}}} A$. Thus, to a circle, we associate $A\otimes_{A\otimes A^{\text{op}}} A$. 
\end{example}

The above example looks like Hochschild homology, but of course it is not: the tensor product is not derived. This can be corrected by having field theories that are functors from homotopical categories to homotopical categories. 

In \cite{baez_dolan} Baez and Dolan propose a higher categorical variant of this definition called ``extended topological field theories.'' It would take us too far afield to fully describe this, but both Baez-Dolan and Lurie \cite{LurieFT} provide excellent expositions. Motivated by this, Hopkins and Lurie \cite{LurieFT} provide a classification of these extended field theories in the setting of $(\infty, n)$-categories. These $(\infty, n)$-categories are the weak or homotopical versions of $n$-categories: this is in analogy to how any model fo $(\infty,1)$-catgories is the weak or homotopical version of ordinary categories. For our purposes, we will only need $(\infty,1)$-categories; this is exactly the setting where the example above will yield honest Hochschild homology.  In this one-dimensional case, the Hopkins-Lurie classification roughly says that given a symmetric monoidal functor between $(\infty,1)$-categories, $F: \mbf{Cob}^{\amalg} \to \mc{C}^{\otimes}$, such a functor is completely determined by its value at a point (much as in the one-dimensional case of vector spaces above). This allows for the following example:

\begin{example}
Let $\mbf{BiMod}$ be the $\infty$-category of bimodules (this can be bimodules in spectra, DGAs, or any other symmetric monoidal $\infty$-category). Then by the Hopkins-Lurie classification, a topological field theory $F: \mbf{Cob}^{\amalg} \to \mc{C}^{\otimes}$ is completely determined by its value at a point, $F(+) = A$. In this case, $F(-) = A^{\text{op}}$ and $F(S^1) = A\otimes^{\mbf{L}}_{A\otimes A^{\text{op}}} A$. Thus, in the case when have as our target category bimodules in spectra, $F(S^1) = \thh(A)$, where $A$ is the value taken by the field theory at a point. If we take bimodules in DGAs, $F(S^1) = \operatorname{HC}_\ast (A)$ where $\operatorname{HC}$ denotes the Hochschild chains. 
\end{example}

We see that topological Hochschild homology has a fascinating relationship with field theories: it is in some sense easiest example of a manifold invariant that comes from field theories. This provides a way to interpret the main result of this paper in terms of field theories. 

\begin{example}
Given a one-dimensional topological field theory $F: \mbf{Cob}^{\amalg} \to \mbf{BiMod}$, we demonstrated above that if $F(+) = A$, then $F(S^1) = \thh (A)$. We may then ask the following question. Is there a field theory whose value on a circle is dual to $\thh(A)$? If $A$ has a derived Koszul dual, $B$, then our main theorem answers that question in the affirmative. The field theory $\widetilde{F}$ such that $\widetilde{F}(+) = B^{\text{op}}$ will give a field theory whose value on a circle is dual to that of $F$. 
\end{example}

Thus, \ref{main_thm} provides a kind of duality for low-dimensional field theories. 

The rest of this section will be devoted to conjecturing an extended version of this result. 

We pause here to note the existence of an invariant called topological chiral homology \cite{LurieHA, AyalaFrancis, andrade}. This is a theory which takes as input an $\mbf{E}_n$-algebra in an $(\infty,1)$-categorry $\mc{C}$ and an $n$-manifold, $M$ and outputs an element of $\mc{C}$ which is denoted by $\int_M A$. It also turns out that when $A$ is an $\mbf{E}_1$-algebra and $M = S^1$ that 
\[
\thh(A) = \int_{S^1} A. 
\]
Given that, let us write the theorem proved above in more suggestive notation:
\[
\int_{S^1} A \simeq D \left( \int_{S^1} \mf{D}A \right).
\]

Furthermore, in an appropriate target category, topological chiral homology is a topological field theory \cite{LurieFT}. 

Ralph Cohen has made the following conjecture

\begin{conj}\label{chiral_homology_conjecture}
Let $A$ be an $\mbf{E}_n$-algebra spectrum and let $M$ be an $n$-manifold. Then
\[
\int_M A \simeq D\left(\int_M \mf{D}A\right)
\]
where $\mf{D}A$ denotes the Koszul dual. 
\end{conj}

\begin{rmk}
In fact, this paper grew out of trying to understand this conjecture for low-dimensional cases. 
\end{rmk}

\begin{rmk}
Although something like the above is no doubt true, there are most likely subtle compactness issues that will demand extra hypotheses, as in the case of $\thh$. At this time, however, the author has no idea what kind of compactness conditions would arise in this case. 
\end{rmk}

The above would be an extremely exciting conjecture to prove, as it would provide a full-fledged duality in field theories. 

However, we can make the conjecture slightly more general. Note that $\mbf{E}_0$ Koszul duality is simply what is usually called duality and also, that given an $\mbf{E}_n$-algebra $A$, we can consider it as an $\mbf{E}_k$ algebra for any $k < n$. Then Conjecture~\ref{chiral_homology_conjecture} is saying that the chiral homology of $A$ is equivalent to the $\mbf{E}_0$-Koszul dual of the chiral homology of the $\mbf{E}_n$ Koszul dual of $A$. We also note that if $n < k$ and $A$ is an $\mbf{E}_k$ algebra, then $\int_M A$ is an $\mbf{E}_{k-n}$ algebra. 

The above conjecture can then be refined:

\begin{conj}
Let $A$ be an $\mbf{E}_{n+\ell}$-algebra and $M$ an $n$-manifold. Then
\[
\int_M A \simeq \mf{D}_{\mbf{E}_{\ell}} \left(\int_M \mf{D}_{\mbf{E}_{n+\ell}} A\right)
\]
\end{conj}

\newpage

\bibliographystyle{amsplain}	
\bibliography{kosbib}		
\end{document}